\documentclass[a4paper, twoside]{article}

\usepackage{amsmath}
\usepackage{amsthm}
\usepackage{amssymb}
\usepackage{amsfonts}
\usepackage{mathrsfs}
\usepackage[all]{xy}
\usepackage{enumerate}

\theoremstyle{plain}

\newtheorem{theorem}{Theorem}[section]
\newtheorem{proposition}[theorem]{Proposition}
\newtheorem{lemma}[theorem]{Lemma}
\newtheorem{corollary}[theorem]{Corollary}
\newtheorem{definition}[theorem]{Definition}

\theoremstyle{definition}

\newtheorem{notation}[theorem]{Notation}

\newtheorem{construction}[theorem]{Construction}
\newtheorem{blank}[theorem]{}

\theoremstyle{remark}

\newtheorem{remark}[theorem]{Remark}
\newtheorem{example}[theorem]{Example}
\newcommand{\bbA}{\mathbb{A}}

\newcommand{\bbC}{\mathbb{C}}
\newcommand{\bbD}{\mathbb{D}}

\newcommand{\bbF}{\mathbb{F}}
\newcommand{\bbG}{\mathbb{G}}

\newcommand{\bbN}{\mathbb{N}}

\newcommand{\bbQ}{\mathbb{Q}}
\newcommand{\bbR}{\mathbb{R}}
\newcommand{\bbS}{\mathbb{S}}

\newcommand{\bbZ}{\mathbb{Z}}
\newcommand{\calA}{\mathcal{A}}
\newcommand{\calB}{\mathcal{B}}
\newcommand{\calC}{\mathcal{C}}
\newcommand{\calD}{\mathcal{D}}

\newcommand{\calG}{\mathcal{G}}

\newcommand{\calM}{\mathcal{M}}
\newcommand{\calN}{\mathcal{N}}
\newcommand{\calO}{\mathcal{O}}
\newcommand{\calP}{\mathcal{P}}

\newcommand{\calS}{\mathcal{S}}
\newcommand{\calT}{\mathcal{T}}
\newcommand{\calU}{\mathcal{U}}
\newcommand{\calV}{\mathcal{V}}

\newcommand{\frkS}{\mathfrak{S}}

\newcommand{\scrH}{\mathscr{H}}

\newcommand{\scrM}{\mathscr{M}}

\newcommand{\scrS}{\mathscr{S}}

\newcommand{\Fp}{\mathbb{F}_p}
\newcommand{\Fq}{\mathbb{F}_q}
\newcommand{\Fbar}{\overline{\mathbb{F}}}
\newcommand{\Zp}{\mathbb{Z}_p}

\newcommand{\Qp}{\mathbb{Q}_p}
\newcommand{\Qq}{\mathbb{Q}_q}

\DeclareMathOperator{\GL}{GL}

\DeclareMathOperator{\GSp}{GSp}

\DeclareMathOperator{\Spec}{Spec}
\DeclareMathOperator{\Sh}{Sh}
\DeclareMathOperator{\Lie}{Lie}
\DeclareMathOperator{\Cent}{Cent}

\DeclareMathOperator{\Hom}{Hom}

\DeclareMathOperator{\im}{im}

\DeclareMathOperator{\Frac}{Frac}

\DeclareMathOperator{\Fil}{Fil}

\begin{document}

\title{The $\mu$-ordinary locus for Shimura varieties of Hodge type}
\author{Daniel Wortmann\thanks{Department of Mathematics, University of Paderborn (Warburger Stra{\ss}e 100, 33098 Paderborn)}\\
        e-mail: \texttt{daniel.wortmann@math.upb.de}
        }
\date{\today}

\maketitle

\begin{abstract}
  We review the Newton stratification and Ekedahl-Oort stratification on the special fiber of a smooth integral model for a Shimura variety of Hodge type at a prime of good reduction. We show that the $\mu$-ordinary locus coincides with the generic Ekedahl-Oort stratum, and that for any two geometric points in the $\mu$-ordinary locus there is an isomorphism of the attached Dieudonn{\'e} modules with additional structure. As a consequence, we proof that the $\mu$-ordinary locus is open and dense, thus generalizing the results which were already known in the PEL-case. To prove our results we provide a method which allows to reduce the equality of strata to a group theoretic statement.
\end{abstract}

\tableofcontents

\section{Introduction}

Let $G$ be a connected reductive group over $\bbQ$ and let $X$ be a $G(\bbR)$-conjugacy class of homomorphisms $\bbS\to G_{\bbR}$ such that the pair $(G,X)$ is a Shimura datum. Then for any sufficiently small open and compact subgroup $K\subseteq G(\bbA_f)$ the associated Shimura variety $\Sh_K(G,X):=G(\bbQ)\setminus X\times G(\bbA_f)/K$ is a smooth, projective complex variety and admits a canonical model over the reflex field $E$ of the Shimura datum $(G,X)$. 

Let $p$ be a prime number. Suppose that $K_p$ is a hyperspecial subgroup of $G(\Qp)$, and let $\Sh_{K_p}(G,X):=\varprojlim_{K^p}\Sh_{K_pK^p}(G,X)$, where the limit is taken over open compact subgroups $K^p\subseteq G(\bbA_f^p)$. In \cite{La} Langlands suggested that $\Sh_{K_p}(G,X)$ should have an integral canonical model $\scrS_{K_p}(G,X)$ over the local ring $o_{E,(v)}$ at any place $v$ of $E$ lying above $p$, this conjecture was later refined by Milne in \cite{Mi1} (see Def. \ref{CanModelDef} for the precise notion of an integral canonical model). In particular, if $\scrS_{K^p}(G,X)$ exists, then for any open compact $K=K_pK^p\subseteq G(\bbA_f)$ (with $K^p$ sufficiently small) the quotient $\scrS_K(G,X):=\scrS_{K_p}(G,X)/K^p$ is a smooth model for the Shimura variety $\Sh_K(G,X)$ over $O_{E,(v)}$.

Consider the classical example of a symplectic Shimura datum, associated to a symplectic rational vector space $(V,\psi)$. 
In this case $G=\GSp(V,\psi)$ and the models $\scrS_K(G,S^{\pm})$ are given as moduli spaces over $\Spec(\bbZ_{(p)})$ of principally polarized abelian schemes of relative dimension $n:=\dim_{\bbQ}(V)/2$ together with a mod-$K$ level structure, so the special fibers of such models can be studied via this moduli interpretation: There is a universal abelian scheme $\calA\to\scrS_K(G,S^{\pm})$, and thus for any algebraically closed field $k$ of characteristic $p$ a point $x\in\scrS_K(G,S^{\pm})(k)$ defines an abelian variety $\calA_x$ of dimension $n$ over $k$. The classification of the associated $p$-divisible groups $\calA_x[p^{\infty}]$ up to isogeny then gives the Newton polygon stratification of $\scrS_K(G,S^{\pm})\otimes\Fp$, which has been studied by Oort (\cite{Oo1}), de Jong-Oort (\cite{dJO}), and many others. On the other hand, the Ekedahl-Oort stratification (or EO-stratification) on $\scrS_K(G,S^{\pm})\otimes\overline{\Fp}$ is obtained by classifying the $p$-torsion subgroups $\calA_x[p]$ up to isomorphism (as $BT_1$-groups), see \cite{Oo2}. Other examples of stratifications arise from considering $\calA_x[p^{\infty}]$ up to isomorphism (\cite{Oo3}), and from the $p$-rank of $\calA_x[p]$ (\cite{Kob}). The \emph{ordinary locus} in $\scrS_K(G,S^{\pm})\otimes\Fp$ is defined as the set of those points $x$ such that for a geometric point $\hat{x}$ lying over $x$ the $p$-divisible group $\calA_{\hat{x}}[p^{\infty}]$ is isogenous to a product of {\'e}tale and multiplicative groups. It has been known for a long time that the ordinary locus is a dense subset of $\scrS_K(G,S^{\pm})$, see for example Koblitz' proof in \cite{Kob}.



In the case of a PEL-type Shimura variety integral canonical models have been constructed by Kottwitz in \cite{Ko2}. They have an explicit interpretation as moduli spaces over $\Spec(o_{E,(v)})$ of abelian schemes with additional structures, so again one has means to study the special fibers $\scrS_K(G,X)\otimes\kappa(v)$ (where $\kappa(v)$ is the residue class field of $E$ at $v$) via the universal object. However, the naive definitions lead to undesirable results in this context. For example the density theorem for the ordinary locus fails, for the simple reason that the ordinary locus may be empty. This has led to refined stratifications which pay respect to the additional structures on the abelian varieties in question (see e.g. \cite{RR} for the Newton stratificaton, \cite{VW} for the EO-stratification). The $\mu$-\emph{ordinary locus} of $\scrS_K(G,X)\otimes\kappa(v)$ is here defined to be the most general Newton stratum, it was shown to be open and dense by Wedhorn in \cite{We1} by a deformation theoretic argument. This theorem was later reproven by Moonen in \cite{Mo2} who showed that the $\mu$-ordinary Newton stratum coincides with the unique open Ekedahl-Oort stratum, which was by then known to be dense (see \cite{We2}). This is one instant of the fruitful interaction of these two statifications, another being the recent work of Viehmann and Wedhorn (see \cite{VW}), who deduce the nonemptiness of all Newton strata from the fact that all Ekedahl-Oort strata are nonempty.

In this article we prove the density of the $\mu$-ordinary locus in the case of a Shimura variety of Hodge type. Although in general we do not have a moduli interpretation of the models any more in this case, there is still a natural abelian scheme $\calA\to\scrS_K(G,X)$ together with a tensor structure on $H_{\mathrm{dR}}^1(\calA/\scrS_K(G,X))$, which allows to define Newton strata and EO-Strata in analogy to the PEL-case. In order to show the density of the $\mu$-ordinary locus we generalize the technique used in \cite{VW} to compare Newtonstrata and Ekedahl-Oort strata with the aid of group theoretic methods. In particular this allows to show that the $\mu$-ordinary locus once again agrees with the unique open Ekedahl-Oort stratum, thereby generalizing Moonen's result from \cite{Mo2}. We remark that the known proofs of these results for PEL-type Shimura varieties often rely on an explicit case-by-case analysis, so that the group theoretic approach also provides a new point of view in this case.\\

Let us explain in more detail the contents of this paper. We fix an embedding $(G,X)\hookrightarrow(\GSp(V,\psi),S^{\pm})$ of Shimura data for some symplectic vector space $(V,\psi)$. There is a connected reductive group scheme $\calG$ over $\Zp$ such that $K_p=\calG(\Zp)\subseteq G(\Qp)$. The existence of integral canonical models for $\Sh_K(G,X)$ (where $K=K_pK^p$ for $K^p$ sufficiently small) was shown by Kisin in \cite{Ki1} (with some restrictions for $p=2$). The starting point of the construction is the observation that one may choose a lattice $\Lambda\subseteq V$ and a finite set of tensors $s$ over $\Lambda_{\bbZ_{(p)}}$ such that $\calG\subseteq\GL(\Lambda_{\Zp})$ is the stabilizer subgroup of $s_{\Zp}$. The model $\scrS_K(G,X)$ is then defined as the normalization of the closure of $\Sh_K(G,X)$ in a suitable moduli space of abelian schemes over $o_{E,(v)}$, thus by construction $\scrS_K(G,X)$ naturally comes along with an abelian scheme $\calA$ on it. The tensors $s$ can be shown to give rise to tensors $s_{\mathrm{dR}}$ over $H_{\mathrm{dR}}^1(\calA\otimes E/\Sh_K(G,X))$. The key step in the proof that the schemes $\scrS_K(G,X)$ indeed give an integral canonical model for $\Sh_{K_p}(G,X)$ is now to show that $\scrS_K(G,X)$ is \emph{smooth}. The technique used in \cite{Ki1} to achieve this shows at the same time that the $s_{\mathrm{dR}}$  extend to tensors $s_{\mathrm{dR}}^{\circ}$ over $H_{\mathrm{dR}}^1(\calA/\scrS_K(G,X))$, in a way such that for any closed or geometric point $x$ of the special fiber of $\scrS_K(G,X)$ one gets induced tensors over the contravariant Dieudonn{\'e} module $\bbD(\calA_x[p^{\infty}])$ which are Frobenius invariant and define a subgroup isomorphic to $\calG_{W(\kappa(x))}$. 

To make this more precise, let us introduce some more notation: Let $\Fbar$ be an algebraic closure of $\Fp$. For simplicity, throughout the rest of the introduction we restrict ourselves to $\Fbar$-valued points. In fact, to introduce the stratifications we also need similar descriptions for points over more general algebraically closed fields and points with finite residue fields. However, since all strata will be locally closed subsets, the $\Fbar$-valued points contain all topological informations on the stratifications, once defined.\\
Let $\calO:=W(\Fbar)$ and  $L:=\mathrm{Frac}(\calO)$, and let $\sigma$ be the Frobenius isomorphism of $\Fbar$ resp. of $\calO$ and $L$. We fix a Borel pair $(\calB,\calT)$ of $\calG$ over $\calO$ which is $\sigma$-invariant. This exists since $\calG$ is quasisplit. Our Shimura datum defines in the usual way a conjugacy class $[\nu]$ of cocharacters for $G$, and hence for $\calG$. We define $\mu$ as the unique dominant cocharacter in $X_*(\calT)$ with respect to $\calB$ such that $\sigma^{-1}(\mu)^{-1}$ lies in $[\nu]$.\\
Let $\Lambda^*$ be the dual $\bbZ$-module of $\Lambda$. The tensors $s$ can also be viewed as tensors over $\Lambda^*_{\bbZ_{(p)}}$ in a canonical way. We let $\calG$ act on $\Lambda^*_{\Zp}$ via the contragredient representation $\GL(\Lambda)\to\GL(\Lambda^*),\ g\mapsto g^{\vee}:=(g^{-1})^*$. 

Now let $x$ be an $\Fbar$-valued point of $\scrS_K(G,X)$, and let $(\bbD_x,F,V):=\bbD(\calA_x[p^{\infty}])$ be the associated contravariant Dieudonn{\'e} module over $\calO$, then the following hold (Corollary \ref{CrysTensors}, Lemma \ref{LinearizationLem}):
\begin{enumerate}[(1)]
  \item The tensors $s_{\mathrm{dR}}^{\circ}$ induce $F$-invariant tensors $s_{\mathrm{cris},x}$ on $\bbD_x$, and there is an isomorphism of $\calO$-modules $\Lambda^*_{\calO}\simeq\bbD_x$ which identifies $s_{\calO}$ with $s_{\mathrm{cris},x}$.
  \item If we identify $\bbD_x$ with $\Lambda^*_{\calO}$ using an isomorphism as in (1), then $F=g^{\vee}(1\otimes\sigma)$ for some $\calG(\calO)\mu(p)\calG(\calO)$, and this element is independent of the choice of the isomorphism up to $\sigma$-conjugation by an element of $\calG(\calO)$.
\end{enumerate}
Thus, writing $\calC(\calG,\mu)$ for the set of $\calG(\calO)$-$\sigma$-conjugacy classes of the double coset $\calG(\calO)\mu(p)\calG(\calO)\subseteq G(L)$, we obtain a well-defined map 
\[
  \gamma\colon \scrS_K(G,X)(\Fbar)\longrightarrow\calC(\calG,\mu)
\]
by sending $x$ to the conjugacy class of the element $g$.

Properties (1) and (2) are direct consequences of the results in \cite{Ki1}, though they are not explicitly stated there. We refer the reader also to \S 1 of the recent preprint \cite{Ki2}. The dual lattice appears here due to the fact that we use contravariant Dieudonn{\'e} theory, which is also the reason for our definition of $\mu$.\\

The Newton stratification is easily described in this context: We have a natural map $\tilde{\theta}\colon\calC(\calG,\mu)\to B(G)$, where $B(G)$ denotes the set of $\sigma$-conjugacy classes in $G(L)$. The Newton strata are then given as the fibers $\calN^b=\theta^{-1}(\{b\})$ of the composite map
\[
  \theta\colon\scrS_K(G,X)(\Fbar)\stackrel{\gamma}{\longrightarrow}\calC(\calG,\mu)\stackrel{\tilde{\theta}}{\longrightarrow} B(G).
\]
Equivalently, two points $x,x'\in\scrS_K(G,X)(\Fbar)$ lie in the same Newton stratum if and only if there is an isomorphism of isocrystals $\bbD_x\otimes_{\calO}L\simeq\bbD_{x'}\otimes_{\calO}L$ which respects the tensors $s_{\mathrm{cris},x}$ and $s_{\mathrm{cris},x'}$. These strata $\calN^b$ are in fact already defined over $\kappa(v)$ (see Section \ref{NewtStratSec} for the precise definition), and a result of Vasiu (\cite{Va1}, 5.3.1.) shows that they are locally closed subsets of $\scrS_K(G,X)\otimes\kappa(v)$. The image of $\tilde{\theta}$ is the subset $B(G,\mu)\subseteq B(G)$ which has already been considered in the context of PEL-type Shimura varieties and affine Deligne-Lusztig sets. It is endowed with a partial order $\preceq$, and contains a unique maximal element $b_{\mathrm{max}}\in B(G,\mu)$ with respect to this order. We define the $\mu$-\emph{ordinary locus} in $\scrS_K(G,X)\otimes\kappa(v)$ as the Newton stratum $\calN^{b_{\mathrm{max}}}$.\\

We can now state our first main theorem:
\begin{theorem}
\label{MainThm1}
  The $\mu$-ordinary locus is open and dense in $\scrS_K(G,X)\otimes\kappa(v)$.
\end{theorem}

To prove this, we relate the Newton stratification to the Ekedahl-Oort stratification on $\scrS_K(G,X)\otimes\Fbar$, which has been defined by C. Zhang in \cite{Zh1}. Just as in the case of a PEL-type Shimura variety the definition of this stratification relies on the theory of $\calG_{\Fp}$-zips. As the precise construction is somewhat involved, we only state the main results here, see Section \ref{EOStratSec} for details. 
Let $(W,S)$ be the Weyl group of $\calG$ with respect to $(\calB,\calT)$, and let $J\subseteq S$ be the type of the cocharacter $\sigma(\mu)$. In analogy to the PEL-case the Ekedahl-Oort stratification is then parametrized by the set ${^J}W$ of shortest left-coset representatives for $W_J$ in $W$. This set carries a partial order $\preceq$, which refines the Bruhat order, and there is again a unique maximal element $w_{\mathrm{max}}\in{^J}W$ with respect to $\preceq$. By the main results of \cite{Zh1}, each stratum $\calS^w\subseteq\scrS_K(G,X)\otimes\Fbar$ is locally closed, and the closure of $\calS^w$ is precisely the union of the strata $\calS^{w'}$ with $w'\preceq w$. Since this in particular implies that $\calS^{w_{\mathrm{max}}}$ is open and dense in $\scrS_K(G,X)\otimes\Fbar$, Theorem \ref{MainThm1} follows directly from our second main result, which generalizes Moonen's description of the $\mu$-ordinary locus in the PEL-case: 

\begin{theorem}
\label{MainThm2}
  The strata $\calN^{b_{\mathrm{max}}}$ and $\calS^{w_{\mathrm{max}}}$ are equal as subsets of $\scrS_K(G,X)\otimes\Fbar$. Furthermore, for any two $\Fbar$-valued points $x,x'$ in this set there is an isomorphism of Dieudonn{\'e} modules $\bbD_x\simeq\bbD_{x'}$ which identifies $s_{\mathrm{cris},x}$ with $s_{\mathrm{cris},x'}$.
\end{theorem}

To show this, we work with $\Fbar$-valued points and make use of the map $\gamma\colon\scrS_K(G,X)(\Fbar)\to\calC(\calG,\mu)$ considered above: By definition of the Ekedahl-Oort stratification, the $\Fbar$-valued points of the strata $\calS^w$ are the fibers of a map $\zeta\colon\scrS_K(G,X)(\Fbar)\to{^J}W$. We show that this map factors via $\gamma$, giving rise to a commutative diagram
\[
  \begin{xy}
     \xymatrix{
                                                                       & &                   & B(G,\mu)\\
       \scrS_K(G,X)(\Fbar)\ar[rrru]^{\theta}\ar[rr]^{\gamma}\ar[rrrd]_{\zeta} & & \calC(\calG,\mu) \ar[ru]_{\tilde{\theta}}\ar[rd]^{\tilde{\zeta}} &    \\
                                                                       & &                   & {^JW}
     }
  \end{xy}
\]
which allows to some extend to compare the stratifications inside the set $\calC(\calG,\mu)$. This is in analogy to the method already used in \cite{VW} for the PEL-case. In particular, the diagram shows that Theorem \ref{MainThm2} follows once we know that $\tilde{\theta}^{-1}(\{b_{\mathrm{max}}\})=\tilde{\zeta}^{-1}(\{w_{\mathrm{max}}\})$ in $\calC(\calG,\mu)$, and that this set consists of a single element. Since we can give a precise description of the fibers of $\tilde{\theta}$ and $\tilde{\zeta}$, we are reduced to a purely group theoretic result for the group $\calG$. 

We finally prove this to be true, in a more general context, in the final section of this article, which is logically independent of the rest of the paper. \\

\noindent
{\bf Acknowledgements.}
I wish to thank my advisor T. Wedhorn deeply for his continuous encouragement and interest in this work. Further I am grateful to E. Lau for helpful discussions on Dieudonn{\'e} theory and much useful advice, and to J.-S. Koskivirta for many helpful comments.

\section{General notations and conventions}

\begin{blank}
\label{sigmaPrep}
  For a perfect field $k$ of positive characteristic $p$ we write $W(k)$ for the Witt ring over $k$, and $L(k)$ for its quotient field. We generally denote by $\sigma$ the Frobenius automorphism $a\mapsto a^p$ of $k$ (with the exception of the last section, where $\sigma$ will denote a finite power of this map), and also its lift to $W(k)$ and $L(k)$. 
  
  Let $k$ be a perfect field of characteristic $p$. Let $R$ be either $k$ or $W(k)$, and let $R_0\subseteq R$ be the subring of elements which are fixed by $\sigma$ (i.e., either $R_0=\Fp$ or $R_0=\Zp$). For any $R$-module $M$ let $M^{(\sigma)}:=M\otimes_{R,\sigma}R$, and for a homomorphism $\beta\colon M\to N$ of $R$-modules write $\beta^{(\sigma)}:=\beta\otimes 1\colon M^{(\sigma)}\to N^{(\sigma)}$. If $f\colon M\to N$ is a $\sigma$-linear map of $R$-modules then
  \[
    M^{(\sigma)}\longrightarrow N,\quad m\otimes a\longmapsto af(m)
  \]
  is $R$-linear, and if $f$ is $\sigma^{-1}$-linear then
  \[
    M\longrightarrow N^{(\sigma)},\quad m\longmapsto f(m)\otimes 1  
  \]
  is $R$-linear. In both cases we call the resulting homomorphism the \emph{linearization} of $f$ and denote it by $f^{\mathrm{lin}}$.
  
  Now let $M_0$ be an $R_0$-module, and let $M=M_0\otimes_{R_0}R$. Then $\sigma$ and $\sigma^{-1}$ act on $M$ via $1\otimes\sigma$ and $1\otimes\sigma^{-1}$ respectively. Further, there is a canonical isomorphism 
  \[
    M=M_0\otimes_{R_0}R\stackrel{\sim}{\longrightarrow} M\otimes_{R_0}R\otimes_{R,\sigma}R=M^{(\sigma)},\quad m\otimes a\mapsto m\otimes 1\otimes a.
  \]
  We will often use this isomorphism to identify $M$ with $M^{(\sigma)}$. For example, if $f\colon M\to N$ is $\sigma$-linear, we also write $f^{\mathrm{lin}}\colon M\cong M^{(\sigma)}\to N$, with this notation we then have that $f=f^{\mathrm{lin}}\circ(1\otimes \sigma)$. 
  
  If $M_0$ is a finitely generated free $R_0$-module, then $\sigma$ also acts on $\GL(M)$ and on the group of cocharacters $\Hom_R(\bbG_{m,R},\GL(M))$. For $g\in\GL(M)$ we have $\sigma(g)=(1\otimes\sigma)\circ g\circ (1\otimes\sigma^{-1})$, and for a cocharacter $\lambda\colon\bbG_{m,R}\to\GL(M)$ we find that $\sigma(\lambda)(a)=\sigma(\lambda(a))$ for all $a\in R$.  
\end{blank}

\begin{blank}
\label{TensorPrep}
  Let $R$ be any ring. If $M$ is a finitely generated free module over $R$, we denote by $M^{\otimes}$ the direct sum of all $R$-modules that arise from $M$ by applying the operations of taking duals, tensor products, symmetric powers and exterior powers a finite number of times. An element of $M^{\otimes}$ will be called a \emph{tensor} over $M$. We have an obvious notion of base change for tensors. Let $M^*$ be the dual $R$-module of $M$. Since there is a canonical identificaton of $M^{\otimes}$ with $(M^*)^{\otimes}$ we can view tensors over $M$ as tensors over $M^*$ as well.
  
  Let $M$ and $M'$ be finitely generated free $R$-modules and let $s=(s_i)_{i\in I}$ and $s'=(s_i')_{i\in I}$ be families of tensors over $M$ and $M'$ respectively. Every isomorphism $f\colon M\to M'$ gives an isomorphism $(f^{-1})^*\colon (M)^*\to (M')^*$ and thus $f^{\otimes}\colon M^{\otimes}\to(M')^{\otimes}$. We will write $f\colon (M,s)\to (M',s')$ if and only if $f^{\otimes}$ takes $s_i$ to $s_i'$ for all $i\in I$. We say that a family of tensors $(s_i)_{i\in I}$ over $M$ \emph{defines} the subgroup $G\subseteq \GL(M)$ if
  \[
    G(R')=\{g\in\GL(M_{R'})\mid g^{\otimes}((s_i)_{R'})=(s_i)_{R'}\text{ for all }i\in I\}
  \]
  for every $R$-algebra $R'$. We have the contragredient representation
  \[
    (\cdot)^{\vee}\colon\GL(M)\longrightarrow\GL(M^*),\quad g\longmapsto g^{\vee}:=(g^{-1})^*,
  \]
  which is in fact an isomorphism of group schemes over $R$. Let $(s_i)_{i\in I}$ be a family of tensors over $M$, defining a subgroup $G\subseteq\GL(M)$. Then these tensors $(s_i)_{i\in I}$, when we consider them as tensors over $M^*$, define the subgroup $\{g^{\vee}\mid g\in G\}\subseteq\GL(M^*)$.
\end{blank}

\section{Shimura data of Hodge type and the tower of Shimura varieties}
\label{ShimVarSec}

Let $G$ be a connected reductive group over $\bbQ$ and let $X$ be a $G(\bbR)$-conjugacy class of algebraic morphisms $\bbS\to G_{\bbR}$ such that $(G,X)$ is a Shimura datum of Hodge type. By definition, this means that there is an embedding $(G,X)\hookrightarrow(\GSp(V,\psi),S^{\pm})$ into a symplectic Shimura datum, which we fix once and for all. We will often simply write $\GSp(V)$ for $\GSp(V,\psi)$ with the symplectic pairing implied.

The datum $(G,X)$ defines conjugacy classes of cocharacters for $G$ as follows: Every element $h\in X$ defines a Hodge decomposition $V_{\bbC}=V^{(-1,0)}\oplus V^{(0,-1)}$ via the embedding $X\hookrightarrow S^{\pm}$.
\begin{definition}
\label{CocharDef}
  \begin{enumerate}[(i)]
    \item We define $\nu_h$ to be the cocharacter of $G_{\bbC}$ such that $\nu_h(z)$ acts on $V^{(-1,0)}$ through multiplication by $z$ and on $V^{(0,-1)}$ as the identity. 
    \item We denote by $[\nu]$ the unique $G(\bbC)$-conjugacy class which contains all the cocharacters $\nu_h$, and by $[\nu^{-1}]$ the conjugacy class which contains the $\nu_h^{-1}$.
  \end{enumerate}
\end{definition}
The reflex field $E$ of $(G,X)$ is defined as the field of definition of $[\nu]$ (or equivalently of $[\nu^{-1}]$), this is known to be a finite extention of $\bbQ$.

We fix a prime number $p$ such that $G$ is of good reduction at $p$. Let $K_p\subseteq G(\Qp)$ be a hyperspecial subgroup. Consider subgroups of the type $K=K_pK^p\subseteq G(\bbA_f)$, where $K^p\subseteq G(\bbA_f^p)$ is open and compact. If $K^p$ is sufficiently small, then the double quotient
\[
  \Sh_K(G,X):=G(\bbQ)\setminus X\times G(\bbA_f)/K
\]
(where $G(\bbQ)$ acts diagonally and $K$ acts on the right factor) has a natural structure as a smooth quasi-projective variety over $\bbC$, and further this variety has a canonical model over $E$. In the sequel we will always view $\Sh_K(G,X)$ as an algebraic variety over $E$.\\

The projective limit
\[
  \Sh_{K_p}(G,X):=\varprojlim_{K^p} \Sh_{K_pK^p}(G,X),
\]
taken over the set of open and compact subgoups of $G(\bbA_f^p)$, carries a contiuous right action of $G(\bbA_f^p)$ in the sense of Deligne (see \cite{Mi1}, 2.1.): Elements $g\in G(\bbA_f^p)$ act by isomorphisms $\Sh_{K_p(gK^pg^{-1})}(G,X)\to\Sh_{K_pK^p}$ such that every $g\in K^p$ gives the identity map on $\Sh_{K_pK^p}(G,X)$, and such that for every normal subgroup $K'^p\subseteq K^p$ the natural covering map induces an isomorphism $\Sh_{K_pK'^p}(G,X)/(K^p/K'^p)\simeq\Sh_{K_pK^p}(G,X)$. In particular, we have an equality $\Sh_{K_pK^p}(G,X)=\Sh_{K_p}(G,X)/K^p$ for every open and compact $K^p\subseteq G(\bbA_f^p)$.\\

We fix a place $v$ of $E$ over $p$. The existence of the hyperspecial subgroup $K_p$ implies that $E$ is unramified at $p$ (\cite{Mi2}, 4.7.). Let $o_E$ be the ring of integers in $E$, and let $o_{E,(v)}$ be its localization at $v$. 

\begin{definition}[\cite{Mi1}, \S 2]
  \label{CanModelDef}
  An \emph{integral canonical model} of $\Sh_{K_p}(G,X)$ over $o_{E,(v)}$ is a projective system $\scrS_{K_p}(G,X)=\varprojlim_{K^p}\scrS_{K_pK^p}(G,X)$ of schemes over $o_{E,(v)}$, indexed by the set of open and compact subgroups of $G(\bbA_f^p)$, together with a continuous right action of $G(\bbA_f^p)$ such that:
  \begin{enumerate}[(i)]
    \item If $K^p$ is sufficiently small, then $\scrS_{K_pK^p}(G,X)$ is smooth over $o_{E,(v)}$ and $\scrS_{K_pK'^p}(G,X)\to \scrS_{K_pK^p}(G,X)$ is {\'e}tale for every $K'^p\subseteq K^p$.
    \item $\scrS_{K_p}(G,X)\otimes_{o_{E,(v)}} E$ is $G(\bbA_f^p)$-equivariantly isomorphic to $\Sh_{K_p}(G,X)$.
    \item Let $Y$ be a regular, formally smooth $o_{E,(v)}$-scheme. Then every morphism $Y\otimes_{o_{E,(v)}} E\to \scrS_{K_p}(G,X)\otimes_{o_{E,(v)}} E$ extends to a morphism $Y\to \scrS_{K_p}(G,X)$.
  \end{enumerate}
\end{definition}

Note that in the situation of (iii) the extension $Y\to \scrS_{K_p}(G,X)$ is automatically unique, since $Y$ is reduced and $Y\otimes_{o_{E,(v)}} E$ is dense in $Y$. Hence a model in the sense of Def. \ref{CanModelDef} is unique up to canonical isomorphism which justifies the name "canonical model". In \cite{Mi1} Milne conjectured that an integral canonical model of $\Sh_{K_p}(G,X)$ always exists (for a general Shimura datum, not necessarily of Hodge type), see also the treatment in (\cite{Mo1}, \S 3).

\begin{example}
  \label{ExPEL1}
  Consider a Shimura datum $(G,X)\hookrightarrow(\GSp(V,\psi),S^{\pm})$ of PEL-type (here $G$ is not connected in general): Let $(B,*)$ be a finite dimensional semi-simple $\bbQ$-algebra with involution which acts on $V$ such that $\psi(bv,w)=\psi(v,b^*w)$ for all $b\in B$ and all $v,w\in V$ and such that
  \[
    G(R)=\{g\in\GSp(V_R,\psi_R)\mid g(bv)=bg(v)\text{ for }b\in B_R, v\in V_R\}
  \]
  for any $\bbQ$-algebra $R$. Then $p$ is a prime of good reduction for $G$ if and only if $B_{\Qp}$ is unramified. In this case, it is shown in (\cite{Ko2}, \S 5 and \S 6) that a canonical integral model exists if $p\geq3$ or if $p=2$ and $G^{\mathrm{ad}}$ has no factor of Dynkin type $D$. The schemes $\scrS_{K_pK^p}(G,X)$ then have an explicit description as a moduli space of abelian schemes with additional structures over $o_{E,(v)}$. 
\end{example}

\section{Integral canonical models for Shimura varieties of Hodge type}

In this section we briefly describe the construction of the canonical integral model for $\Sh_{K_p}(G,X)$, following Kisin's proof in \cite{Ki1}, and introduce the objects which are fundamental for the study of the closed fiber which follows. In the case $p=2$ two restrictions arise in order for the construction to work.

\subsection{Construction of the integral models}
\label{IntModConstrSec}

Let $\calG$ be a reductive model of $G$ over $\bbZ_p$ such that $K_p=\calG(\bbZ_p)$. If $p=2$, we assume that $G^{\mathrm{ad}}$ has no factor of Dynkin type $B$. Then there is a lattice $\Lambda\subseteq V$ and a finite set of tensors $s:=(s_i)\subset\Lambda_{\bbZ_{(p)}}^{\otimes}$ such that $\calG$ is identified with the subgroup of $\GL(\Lambda_{\Zp})$ defined by $s_{\bbZ_p}\subset\Lambda_{\bbZ_p}^{\otimes}$ (\cite{Ki1}, 2.3.1., 2.3.2.) via our chosen embedding $G\hookrightarrow\GSp(V)$. Possibly passing to a homothetic lattice, we may and will further assume that the symplectic pairing $\psi$ on $V$ restricts to a pairing $\Lambda\times\Lambda\to\bbZ$. Note however that $\Lambda$ will not be self-dual with respect to $\psi$ in general. 

Let $\tilde{K}_p$ be the stabilizer of $\Lambda_{\bbZ_p}$ in $\GSp(V)(\bbQ_p)$. Then $K_p=\tilde{K}_p\cap G(\bbQ_p)$. Let $K^p\subseteq G(\bbA_f^p)$ be an open and compact subgroup such that $K:=K_pK^p$ leaves $\Lambda_{\hat{\bbZ}}$ stable (which is the case for all sufficiently small $K^p$). It can be shown that there is an open and compact subgroup  $\tilde{K}^p\subseteq\GSp(V)(\bbA_f^p)$ which contains $K^p$, such that $\tilde{K}:=\tilde{K}_p\tilde{K}^p$ also leaves $\Lambda_{\hat{\bbZ}}$ stable and such that the natural map $\Sh_K(G,X)\rightarrow\Sh_{\tilde{K}}(\GSp(V),S^{\pm})\otimes_{\bbQ}E$ is a closed embedding (\cite{Ki1}, 2.1.2., 2.3.2.). We call a subgroup $\tilde{K}^p$ with these properties \emph{admissible} for $K^p$. If $K'^p\subseteq K^p$ then there is an open and compact subgroup $\tilde{K}'^p\subseteq\tilde{K}^p$ which is admissible for $K'^p$ and we obtain a commutative diagram
\[
  \label{CoupleCD}
  \begin{xy}
  \xymatrix{
  \Sh_{K'}(G,X)\ar@^{(->}[r]\ar[d] & \Sh_{\tilde{K}'}(\GSp(V),S^{\pm})\otimes_{\bbQ}E\ar[d]\\
  \Sh_K(G,X)\ar@^{(->}[r]          & \Sh_{\tilde{K}}(\GSp(V), S^{\pm})\otimes_{\bbQ}E
  }
  \end{xy}
\]
where the horizontal arrows are closed embeddings.


\begin{construction}
\label{ModuliConstr}

We denote by $\Lambda'$ the dual lattice of $\Lambda$ with respect to $\psi$. Let $|\Lambda'/\Lambda|=d$, and let $\dim(V)=2n$. Let $K^p\subseteq G(\bbA_f^p)$ be an open and compact subgroup, and let $\tilde{K}^p$ be admissible for $K^p$. With respect to $\Lambda$, we consider the moduli space $\scrM_{n,d,\tilde{K}^p}$ over $\bbZ_{(p)}$ which parametrizes abelian schemes with a polarization of degree $d$ and a mod-$\tilde{K}^p$ level structure up to isomorphism (see \cite{Ki1}, 2.3.3.). By the classical result of Mumford, $\scrM_{n,d,\tilde{K}^p}$ is representable by a quasi-projective scheme over $\bbZ_{(p)}$ if $\tilde{K}^p$ is sufficiently small. 

Let again $\tilde{K}=\tilde{K}_p\tilde{K}^p$. Due to the moduli interpretation of Shimura varieties of Siegel type, there is an embedding 
\[
  \label{ModuliMap}
  \Sh_{\tilde{K}}(\GSp(V),S^{\pm})\hookrightarrow \scrM_{n,d,\tilde{K}^p}
\]
of $\bbZ_{(p)}$-schemes. We give a description of this map on $\bbC$-valued points, cf. (\cite{Va1}, 4.1.): Let 
\[
  [h,g]\in\Sh_{\tilde{K}}(\GSp(V),S^{\pm})(\bbC)=\GSp(\bbQ)\setminus S^{\pm}\times\GSp(\bbA^f)/\tilde{K}.
\]
Let $V_{\bbC}=V^{(-1,0)}\oplus V^{(0,-1)}$ be the Hodge decomposition induced by $h$. There is a unique $\bbZ$-lattice $\Lambda_g\subset V$ such that $(\Lambda_g)_{\hat{\bbZ}}=g(\Lambda_{\hat{\bbZ}})$ and a unique $\bbQ^{\times}$-multiple $\psi_{h,g}$ of $\psi$ such that $g(\Lambda'_{\hat{\bbZ}})$ is the dual lattice of $g(\Lambda_{\hat{\bbZ}})$ with respect to $\psi_{h,g}$ and such that the form $(v,w)\mapsto\psi_{h,g}(v,h(i)w)$ is positive definite on $V_{\bbR}$. Then $[h,g]$ is mapped to the isomorphism class of $(A, \lambda,\eta)$, where $A:=V^{(-1,0)}/\Lambda_g$, endowed with the polarization $\lambda$ induced by $\psi_{h,g}$, is the polarized complex abelian variety associated to $(V,\psi_{h,g},\Lambda_g,h)$ via Riemann's theorem (see \cite{De1}, 4.7.), and $\eta$ is the right $\tilde{K}^p$-coset of
\[
  \Lambda_{\hat{\bbZ}^p}\stackrel{g^p}{\longrightarrow}g^p(\Lambda_{\hat{\bbZ}^p})=(\Lambda_g)_{\hat{\bbZ}^p}\cong H_1(A,\bbZ)_{\hat{\bbZ}^p}\cong\prod_{l\neq p}T_l(A).
\]
\end{construction}

Recall that $v$ denotes a place of $E$ over $p$, and $o_{E,(v)}$ the localization of $o_E$ at $v$.

\begin{definition}
  \label{IntModelDef}
  Let $K^p\subseteq G(\bbA_f^p)$ be an open and compact subgroup, and let $\tilde{K}^p$ be admissible for $K^p$ such that $\scrM_{n,d,\tilde{K}^p}$ exists as a scheme. Let $K=K_pK^p$ and $\tilde{K}=\tilde{K}_p\tilde{K}^p$. We define $\scrS_K(G,X)$ as the normalization of the closure of $\Sh_K(G,X)$ in $\scrM_{n,d,\tilde{K}^p}\otimes_{\bbZ_{(p)}} o_{E,(v)}$ with respect to the embedding
  \[
    \Sh_K(G,X)\hookrightarrow\Sh_{\tilde{K}}(\GSp(V),S^{\pm})\otimes_{\bbQ}E\hookrightarrow \scrM_{n,d,\tilde{K}^p}\otimes_{\bbZ_{(p)}} o_{E,(v)}.
  \]
\end{definition}

\begin{remark}
\label{ModelIndepRem}
  This definition is indeed independent of the choice of $\tilde{K}^p$: Let $\tilde{K}'^p\subseteq\tilde{K}^p$ be an open and compact subgroup which contains $K^p$ (it is then automatically admissible for $K^p$), then the natural map $\scrM_{n,d,\tilde{K}'^p}\to\scrM_{n,d,\tilde{K}^p}$ is finite and there is a commutative diagram
  \[
    \begin{xy}
    \xymatrix{
      \Sh_K(G,X)\ar@^{(->}[r] \ar@^{(->}[rd] & \scrM_{n,d,\tilde{K}'^p}\otimes_{\bbZ_{(p)}}o_{E,(v)} \ar[d]\\
                                             & \scrM_{n,d,\tilde{K}^p}\otimes_{\bbZ_{(p)}}o_{E,(v)}\ .
    }
    \end{xy}
  \]
   Let $Z$ be a component of $\Sh_K(G,X)$, and denote by $\overline{Z}'$ and $\overline{Z}$ the closures in the $o_{E,(v)}$-schemes on the right hand side of the diagram respectively. The induced map $\overline{Z}'\to\overline{Z}$ is finite and dominant, and is an isomorphism at the generic points. Hence the corresponding map of the respective normalizations is an isomorphism.
\end{remark}

By definition, for every $K^p$ the choice of an admissible $\tilde{K}^p$ gives a natural map $\scrS_K(G,X)\to \scrM_{n,d,\tilde{K}}\otimes_{\bbZ_{(p)}} o_{E,(v)}$, this defines an abelian scheme over $\scrS_K(G,X)$ which is independent of the choice of $\tilde{K}^p$ up to isomorphism by the preceeding remark. If $K'^p\subseteq K^p$ then we have a natural map $\scrS_{K'^p,\tilde{K}'^p}(G,X)\to\scrS_{K^p,\tilde{K}^p}(G,X)$ which is obtained by the choice of suitable admissible subgroups $\tilde{K}'^p\subseteq\tilde{K}^p$.

\begin{theorem}[Kisin, \cite{Ki1} Thm. 2.3.8.]
  \label{CanModelThm}
  If $p=2$, assume that $G^{\mathrm{ad}}$ has no factor of Dynkin type $B$, and that, for each $K^p$, the dual of each abelian variety associated to a point on the special fiber of $\scrS_K(G,X)$ has a connected $p$-divisible group.\\
  Then the following hold:
  \begin{enumerate}[(i)]
    \item $\scrS_K(G,X)$ is a smooth $o_{E,(v)}$-scheme for each $K^p$.
    \item The projective limit $\scrS_{K_p}(G,X):=\varprojlim_{K^p}\scrS_K(G,X)$ is an integral canonical model of $\Sh_{K_p}(G,X)$ over $o_{E,(v)}$ in the sense of Def. \ref{CanModelDef}.
  \end{enumerate}
\end{theorem}

In particular, $\scrS_{K^p}(G,X)$ and hence also $\scrS_K(G,X)=\scrS_{K_p}(G,X)/K^p$ (for $K^p$ sufficiently small) does not depend on the choice of the embedding $(G,X)\hookrightarrow(\GSp(V),S^{\pm})$, nor on the choices made during the construcion.

\subsection{Tensors on the de Rham cohomology}
\label{dRTensorSec}

Although in general we do not have an interpretation of the integral models $\scrS_K(G,X)$ as moduli spaces of abelian schemes with additional structures, each model is by construction naturally endowed with an abelian scheme on it, and the tensors $s$ which define the group $\calG$ induce tensors on de Rham cohomology of this abelian scheme. In this subsection we desribe the construction of these tensors and their relation to $s$, still following \cite{Ki1}. 

We will systematically consider the tensors $s\subset\Lambda_{\bbZ_{(p)}}^{\otimes}$ chosen in the last subsection as tensors over $\Lambda^*_{\bbZ_{(p)}}$ and use the contragredient representation 
\[
  (\cdot)^{\vee}\colon\GL(\Lambda)\stackrel{\sim}{\longrightarrow}\GL(\Lambda^*),
\]
as discussed in Section \ref{TensorPrep}. 

\begin{notation}
In the sequel we will work with a fixed model $\scrS:=\scrS_K(G,X)$ associated to some sufficiently small subgroup $K^p\subseteq G(\bbA_f^p)$ as in Def. \ref{IntModelDef}. We fix an open compact $\tilde{K}^p\subseteq \GSp(V)(\bbA_f^p)$ which is admissible for $K^p$ in the sense of the last subsection. Note that all the constructions below are in fact \emph{independent} of the choice of $\tilde{K}^p$. In the case $p=2$ we assume that the assumptions of Theorem \ref{CanModelThm} hold, so that $\scrS$ is smooth. 

Let $\calA\stackrel{\pi}{\longrightarrow}\scrS$ be the abelian scheme defined by the natural map $\scrS\to\scrM_{n,d,\tilde{K}^p}\otimes_{\bbZ_{(p)}}\calO_{E,(v)}$. Let 
\[
  \calV^{\circ}:=H_{\mathrm{dR}}^1(\calA/\scrS)\quad \text{and}\quad \calV:=H_{\mathrm{dR}}^1(\calA\otimes E/\Sh_K(G,X)).
\]
Then $\calV^{\circ}$ and $\calV$ are locally free modules over $\scrS$ and $\scrS\otimes E=\Sh_K(G,X)$ respectively, and $\calV=\calV^{\circ}\otimes E$. Let $\nabla$ denote the Gau{\ss}-Manin connection on $\calV^{\circ}$ resp. $\calV$. It is known that the Hodge spectral sequence $E_1^{p,q}=R^q\pi_*(\Omega^p_{\calA/\scrS})\ \Longrightarrow\ H^{p+q}_{\mathrm{dR}}(\calA/\scrS)$ degenerates at $E_1$ (\cite{BBM}, 2.5.2.), giving rise to a filtration
\[
  \calV^{\circ}=H_{\mathrm{dR}}^1(\calA/\scrS)\supset\pi_*\Omega^1_{\calA/\scrS}=:\Fil^1\calV^{\circ},
\]
the \emph{Hodge filtration} on $\calV^{\circ}$.
\end{notation}

Let $E'|E$ be any field extension which admits an embedding into $\bbC$, and let $\xi\in\scrS(E')$. Let $\overline{E'}$ be an algebraic closure of $E'$, choose an embedding $\overline{E'}\hookrightarrow\bbC$. We denote by $\bar{\xi}$ and $\xi_{\bbC}$ the $\overline{E'}$-valued and $\bbC$-valued points corresponding to $\xi$. From the embedding $\Sh_K(G,X)\hookrightarrow\scrM_{n,d,\tilde{K}^p}$ used in Constr. \ref{ModuliConstr} we get a natural isomorphism $V\simeq H_1(\calA_{\xi_{\bbC}},\bbQ)$. The dual of this isomorphism maps $s_{\bbQ}\subset(V^*)^{\otimes}$ to a set of tensors over $H^1(\calA_{\xi_{\bbC}},\bbQ)$, and using the comparison isomorphisms
\[
  H^1(\calA_{\xi_{\bbC}},\bbQ)_{\bbC}\cong H^1_{\mathrm{dR}}(\calA_{\xi_{\bbC}}/\bbC),\quad H^1(\calA_{\xi_{\bbC}},\bbQ)_{\bbQ_l}\cong H^1_{\acute{e}t}(\calA_{\xi_{\bbC}},\bbQ_l)\cong H^1_{\acute{e}t}(\calA_{\bar{\xi}},\bbQ_l)
\]
we obtain tensors $s_{\mathrm{dR},\xi}$ on the algebraic de Rham cohomology of $\calA_{\xi_{\bbC}}$ and $s_{\acute{e}t,l,\xi}$ on the $l$-adic {\'e}tale cohomology of $\calA_{\bar{\xi}}$ for every prime number $l$. By a result of Deligne (\cite{De2}, 2.11.), the family $(s_{\mathrm{dR},\xi}, (s_{\acute{e}t,l,\xi}))$ is an absolute Hodge cycle (see loc.cit. \S 2 for the definition of Hodge cycles and absolute Hodge cycles).

\begin{proposition}[\cite{Ki1}, 2.2.1., 2.2.2.]
  \label{dRgeneric}
  \begin{enumerate}[(i)]
    \item For every $\xi\in\scrS(E')$ as above the tensors $s_{\mathrm{dR},\xi}$ are defined over $H^1_{\mathrm{dR}}(\calA_{\xi}/E')$ and the tensors $s_{\acute{e}t,l,\xi}$ are $\mathrm{Gal}(\overline{E'}|E')$-invariant for each $l$.
    \item There exist global sections $s_{\mathrm{dR}}\subset\calV^{\otimes}$ defined over $E$, which are horizontal with respect to the Gau{\ss}-Manin connection $\nabla$, such that the pullback of $s_{\mathrm{dR}}$ to any $\xi\in\scrS(E')$ as above equals the tensors $s_{\mathrm{dR},\xi}\subset(H^1_{\mathrm{dR}}(\calA_{\xi}/E'))^{\otimes}$.
  \end{enumerate}
\end{proposition}  

The extension of the tensors $s_{\mathrm{dR}}$ to sections of $(\calV^{\circ})^{\otimes}$ relies on the following pointwise construction: Let $k$ be a perfect field of finite trancendence degree over $\bbF_p$. Let $W(k)$ be the Witt ring over $k$ and let $L(k):=\mathrm{Frac}(W(k))$. We consider a triple $(\tilde{x},\xi,x)$, where $\tilde{x}$ is a $W(k)$-valued point $\scrS$ and $\xi\in\scrS(L(k))$, $x\in\scrS(k)$ are the corresponding induced points.

Let $\bbD_x$ be the contravariant Dieudonn{\'e} module of the $p$-divisible group of $\calA_x$. Recall that $\bbD_x$ is a free $W(k)$-module together with a $\sigma$-linear map $F$ and a $\sigma^{-1}$-linear map $V$ such that $FV=p=VF$. We have canonical isomorphisms 
\[
  H^1_{\mathrm{dR}}(\calA_{\tilde{x}}/W(k))\cong H^1_{\mathrm{cris}}(\calA_x/W(k))\cong\bbD_x.
\]
By our assumption on $k$, the field $L(k)$ can be embedded into $\bbC$. The choice of an embedding $\overline{L(k)}\hookrightarrow\bbC$ hence yields an absolute Hodge cycle $(s_{\mathrm{dR},\xi}, (s_{\acute{e}t,l,\xi}))$ as above. 

On the other hand, there is also an isomorphism  
\[
  \Lambda_{\Zp}\stackrel{\sim}{\longrightarrow}H_1(\calA_{\xi_{\bbC}},\bbZ)_{\Zp}\cong T_p(\calA_{\xi_{\bbC}})\cong T_p(\calA_{\bar{\xi}}):
\]
With the notations of Constr. \ref{ModuliConstr}, if the $\bbC$-valued point $\xi_{\bbC}$ corresponds to the element $[h,g]\in\Sh_{\tilde{K}}(\GSp(V),S^{\pm})(\bbC)$, then the first arrow is given by $g_p$. Dualizing this isomorphism, and paying respect to the $\mathrm{Gal}(\overline{L(k)}|L(k))$-operation on the right hand side, yields
\[
  \Lambda^*_{\Zp}\simeq T(\calA_{\bar{\xi}})^*(-1)\cong H_{\acute{e}t}^1(\calA_{\bar{\xi}},\Zp),
\]
which sends the tensors $s_{\Zp}\subset(\Lambda^*_{\Zp})^{\otimes}$ to tensors $s_{\acute{e}t,\xi}^{\circ}$ over $H_{\acute{e}t}^1(\calA_{\bar{\xi}},\Zp)$. Since all the isomorphisms involved are compatible, the base change of $s_{\acute{e}t,\xi}^{\circ}$ to tensors over $H_{\acute{e}t}^1(\calA_{\bar{\xi}},\Qp)$ is exactly the $p$-adic component $s_{\acute{e}t,p,\xi}$ of the absolute Hodge cycle defined above. So Prop. \ref{dRgeneric}(i) implies that the $s_{\acute{e}t,\xi}^{\circ}$ are invariant under the action of $\mathrm{Gal}(\overline{L(k)}|L(k))$. Now it follows from Kisin's theory of crystalline representations and $\frkS$-modules that the images of these tensors under the $p$-adic comparison isomorphism
\[
  H_{\acute{e}t}^1(\calA_{\bar{\xi}},\Zp)\otimes_{\Zp}B_{\mathrm{cris}}\stackrel{\sim}{\longrightarrow} H^1_{\mathrm{cris}}(\calA_x/W(k))\otimes_{W(k)}B_{\mathrm{cris}} \cong \bbD_x\otimes_{W(k)}B_{\mathrm{cris}}
\]
are $F$-invariant and are already defined over $\bbD_x$ (\cite{Ki1}, 1.3.6.(1), 1.4.3.(1)). Using the identification $\bbD_x\cong H_{\mathrm{dR}}^1(\calA_{\tilde{x}}/W(k))$ we thus obtain tensors $s_{\mathrm{dR},\tilde{x}}^{\circ}$ over $\calV^{\circ}_{\tilde{x}}$.

\begin{proposition}
  \label{dRExtProp}
  \begin{enumerate}[(i)]
    \item The tensors $s_{\mathrm{dR}}$ of Prop. \ref{dRgeneric} extend (uniquely) to global sections $s_{\mathrm{dR}}^{\circ}\subset(\calV^{\circ})^{\otimes}$ which are horizontal with respect to $\nabla$.
    \item Let $(\tilde{x},\xi,x)$ be a triple as considered above. Then the tensors $s_{\mathrm{dR},\tilde{x}}^{\circ}\subset(\calV^{\circ}_{\tilde{x}})^{\otimes}$ which we obtained via the $p$-adic comparison isomorphism in the above construction are equal to the pullback of $s_{\mathrm{dR}}^{\circ}$ to $\tilde{x}$. 
    \item In the situation of (ii), assume in addition that $k$ is finite or algebraically closed. Then there is a $W(k)$-linear isomorphism
    \[
      (\Lambda^*_{W(k)},s_{W(k)})\stackrel{\sim}{\longrightarrow}(\calV_{\tilde{x}}^{\circ},s_{\mathrm{dR},\tilde{x}}^{\circ})
    \]
    Further, if $\beta$ is any such isomorphism, then there is a cocharacter $\lambda$ of $\calG_{W(k)}$ such that the filtration $\Lambda^*_{W(k)}\supset\beta^{-1}(\Fil^1\calV_{\tilde{x}}^{\circ})$ is induced by $(\cdot)^{\vee}\circ\lambda$.
  \end{enumerate}
\end{proposition}

\begin{proof}
  (i) The existence of $s_{\mathrm{dR}}^{\circ}$ is shown in the proof of (\cite{Ki1}, 2.3.9.). These extensions are automatically unique, since $\scrS$ is in particular an integral scheme and $\calV^{\circ}$ is locally free. By the same reasoning it follows that the $s_{\mathrm{dR}}^{\circ}$ are horizontal with respect to $\nabla$, as they are so over $\scrS\otimes E$.\\
  (ii) If $x$ is a closed point of $\scrS$, then this is immediately clear from the definition of $s_{\mathrm{dR}}^{\circ}$ in (\cite{Ki1}, 2.3.9.). In general, as the equality of tensors in $(\calV^{\circ}_{\tilde{x}})^{\otimes}$ may be tested over $\xi$, the statement amounts to the fact that the $p$-adic comparison isomorphism $H_{\acute{e}t}^1(\calA_{\bar{\xi}},\Qp)\otimes_{\Qp}B_{\mathrm{dR}}\simeq H^1_{\mathrm{dR}}(\calA_{\xi}/L(k)) \otimes_{L(k)}B_{\mathrm{dR}}$ maps the $p$-adic {\'e}tale component of the absolute Hodge cycle $(s_{\mathrm{dR},\xi}, (s_{\acute{e}t,l,\xi}))$ to its de Rham component. If $\calA_{\xi}$ can be defined over a number field, this is a theorem of Blasius and Wintenberger (\cite{Bl1}, 0.3.), and Vasiu (\cite{Va1}, 5.2.16.) observed that their result can also be extended to our more general situation.\\
  (iii) Let $\widetilde{\bbD}$ be the contravariant crystal of the $p$-divisible group $\calA_{\tilde{x}}[p^{\infty}]$ over $W(k)$. Then we have the natural identification
  \[
    \widetilde{\bbD}(W(k))\cong\bbD_x\cong \calV_{\tilde{x}}^{\circ}
  \]
  which is compatible with the Hodge filtrations on both sides, and by (ii) the tensors $s_{\mathrm{dR},\tilde{x}}^{\circ}$ get identified with the images of $s_{\acute{e}t,\xi}^{\circ}$ under the $p$-adic comparison isomorphism. So the first statement of (iii) follows directly from (\cite{Ki1}, 1.4.3. (2)+(3)), applied to the $p$-divisible group $\calA_{\tilde{x}}[p^{\infty}]$ and the tensors $s_{\acute{e}t,\xi}^{\circ}$. Likewise, the proof of (4) in loc.cit. (which proves more than what is claimed) shows that the filtration $\Lambda^*_{W(k)}\supset\beta^{-1}(\Fil^1\calV_{\tilde{x}}^{\circ})$ is induced by a cocharacter of the subgroup of $\GL(\Lambda^*_{W(k)})$ which is defined by the tensors $s_{W(k)}\subseteq(\Lambda^*_{W(k)})^{\otimes}$. As this subgroup is exactly the image of $\calG_{W(k)}$ under $(\cdot)^{\vee}$, the last claim follows.
\end{proof}

We remark that the existence of the tensors $s_{\mathrm{dR}}^{\circ}$ is closely related to the proof of Theorem \ref{CanModelThm}: In fact, the proof of the smoothness of $\scrS$ in (\cite{Ki1}, 2.3.5.) uses a variant of Prop. \ref{dRExtProp}(iii) as a main ingredience, and in turn the arguments given in that proof allow the construction of $s_{\mathrm{dR}}^{\circ}$ in (\cite{Ki1}, 2.3.9.).

\begin{corollary}
  \label{CrysTensors}
  Let $x\in\scrS(k)$, where $k$ is either a finite extension of $\Fp$ or algebraically closed of finite transcendence degree over $\Fp$, and let $\bbD_x$ be the contravariant Dieudonne module of the $p$-divisible group $\calA_x[p^{\infty}]$. Let $\tilde{x}\in\scrS(W(k))$ be a lift of $x$.\\
  Then the images $s_{\mathrm{cris},x}\subset(\bbD_x)^{\otimes}$ of $s_{\mathrm{dR},\tilde{x}}^{\circ}$ via the identification $H_{\mathrm{dR}}^1(\calA_{\tilde{x}}/W(k))\cong\bbD_x$ are independent of the choice of $\tilde{x}$. Further, the tensors $s_{\mathrm{cris},x}$ are $F$-invariant, and there is a $W(k)$-linear isomorphism $(\Lambda_{W(k)}^*,s_{W(k)})\simeq(\bbD_x,s_{\mathrm{cris},x})$.
\end{corollary}

\begin{proof}
  This follows immediately from (i) and (ii) of the last proposition.
\end{proof}

\begin{remark}
  \label{PELRem2}
  In the special case of a PEL-type Shimura variety it is known that one can choose the lattice $\Lambda\subset V$ in a way such that $\Lambda_{\Zp}$ is selfdual with respect to $\psi$ and such that there is a maximal order $o_B$ of $(B_{\Qp},*)$ which acts on $\Lambda_{\Zp}$, and the tensors $s\subset\Lambda_{\Zp}^{\otimes}$ then encode the action of $o_B$ on $\Lambda_{\Zp}$. Due to the interpretation of $\scrS$ as a moduli space of abelian schemes with additional structure, for every $x\in\scrS(k)$ as in Cor. \ref{CrysTensors} there is an action of $o_B$ on the $p$-divisible group $\calA_x[p^{\infty}]$. So in this case Cor. \ref{CrysTensors} is an analogon to the results in (\cite{VW}, \S2) on $p$-divisible groups with PEL-structure. Note however that the authors use covariant Dieudonn{\'e} theory in that article.
\end{remark}

\section{Stratifications of the special fiber}

  Let $\kappa(v)$ be the residue class field of $\calO_{E,(v)}$, and let $\Fbar$ be a fixed algebraic closure of $\Fp$. In this section we define the Newton stratification and the Ekedahl-Oort stratification on the special fiber $\scrS\otimes\kappa(v)$ of $\scrS$ resp. on $\scrS\otimes\Fbar$. These stratifications arise by considering the isocristals resp. Dieudonn{\'e} spaces associated to $\calA_x$ for points $x$ as in Cor. \ref{CrysTensors}, while paying respect to the tensor structure. Just as in the Siegel case and the PEL case, the stratifications are parametrized by combinatorial data which only depends on the Shimura datum $(G,X)$.\\
  
  We start by introducing some group theoretic notions: The reductive group scheme $\calG$ over $\Zp$ is quasisplit and split over a finite {\'e}tale extension of $\Zp$. We fix a Borel subgroup $\calB\subseteq \calG$ and a maximal torus $\calT\subseteq \calB$ which are both defined over $\Zp$. Let $(X^*(\calT),\Phi,X_*(\calT),\Phi^{\vee})$ be the root datum associated to $(\calG,\calT)$ over $\calO$, and let $W$ be the associated Weyl group. The choice of $\calB$ determines a set $\Phi^+\subset\Phi$ of positive roots and a set $S\subset W$ of simple reflections which give $(W,S)$ the structure of a finite Coxeter group. As usual, we call a cocharacter $\lambda\in X_*(\calT)$ dominant, if $\langle \alpha,\lambda\rangle\geq0$ for all $\alpha\in\Phi^+$ (here $\langle\cdot,\cdot\rangle$ is the natural pairing between $X^*(\calT)$ and $X_*(\calT)$). The group $W$ naturally acts on $X_*(\calT)$, and the dominant cocharacters form a full set of representatives for the orbits $W\setminus X_*(\calT)$.
  
  For any local, strictly henselian $\Zp$-algebra $R$ we have a realization of this data with respect to $\calG_R$. In particular $W\cong N_{\calG}(\calT)(R)/\calT(R)$ and the inclusion $X_*(\calT)\cong\Hom_R(\bbG_{m,R},\calT_R)\subseteq\Hom_R(\bbG_{m,R},\calG_R)$ induces a bijection between the quotient $W\setminus X_*(\calT)$ and the set of conjugacy classes of cocharacters for $\calG_R$. If $R\to R'$ is a homomorphism of local and strictly henselian $\Zp$-algebras, then base change to $R'$ yields a bijection between the sets of conjugacy classes of cocharacters for $\calG_R$ and $\calG_{R'}$.\\
  Putting $R=\bbC$, we see that the conjugacy class $[\nu^{-1}]$ from Def. \ref{CocharDef} determines an element of $W\setminus X_*(\calT)$. On the other hand, putting $R=W(k)$ for some algebraically closed field $k$ of characteristic $p$, we obtain an action of the Frobenius $\sigma$ on $X_*(\calT)$, $W$ and $\Phi$. Since $\calB$ and $\calT$ are defined over $\Zp$, this action leaves $S$, $\Phi^+$ and the set of dominant cocharacters stable. 

\begin{definition}
  \label{muDef}
  We define $\mu\in X_*(\calT)$ as the unique dominant element such that $\sigma^{-1}(\mu)\in[\nu^{-1}]$.
\end{definition}

\subsection{The Dieudonn{\'e} module at a geometric point}
  \label{StratPrepSec}
  
  We will use the notations and considerations of Section \ref{sigmaPrep}, and especially apply them to the case that $M_0=\Lambda_{\Zp}$ or $M_0=\Lambda_{\Fp}$. Recall that we use the contragredient representation $(\cdot)^{\vee}\colon \GL(\Lambda)\to\GL(\Lambda^*)$ to let $\calG(R)$ act on $\Lambda^*_R$. 

\begin{lemma}
  \label{CocharLem1}
  Let $k$ be a perfect field of finite transcendence degree over $\Fp$, let $\bar{k}$ be an algebraic closure of $k$. Let $\tilde{x}\in\scrS(W(k))$, and suppose that there exists an isomorphism 
  \[
    \beta\colon(\Lambda^*_{W(k)},s_{W(k)})\stackrel{\sim}{\longrightarrow}(\calV_{\tilde{x}}^{\circ},s_{\mathrm{dR},\tilde{x}}^{\circ}).
  \]
  If $\lambda$ is a cocharacter of $\calG_{W(k)}$ which such that $(\cdot)^{\vee}\circ\lambda$ induces on $\Lambda^*_{W(k)}$ the filtration $\Lambda^*_{W(k)}\supset\beta^{-1}(\Fil^1\calV_{\tilde{x}}^{\circ})$, then $\lambda_{W(\bar{k})}\in[\nu^{-1}]$.
\end{lemma}

\begin{proof}
  Let $\xi\in\scrS(L(k))$ be the generic point of $\tilde{x}$. Let $\overline{L(k)}$ be an algebraic closure of $L(k)$, then we have $W(\bar{k})\hookrightarrow\overline{L(k)}$. So if we choose an embedding $\overline{L(k)}\hookrightarrow\bbC$, it suffices to show that $\lambda_{\bbC}\in[\nu^{-1}]$.
  
  Let $A:=\calA_{\xi_{\bbC}}$. There is a pair $(h,g)\in X\times G(\bbA_f)$ such that
  \[
    \xi_{\bbC}=[h,g]\in\Sh_K(G,X)=G(\bbQ)\setminus X\times G(\bbA_f)/K.
  \]
  If $V_{\bbC}=V^{(-1,0)}\oplus V^{(0,-1)}$ is the Hodge decomposition given by $h$ then, using the notation from Constr. \ref{ModuliConstr}, we have $A\simeq V^{(-1,0)}/\Lambda_g$, and in turn there is an isomorphism $H_1(A,\bbC)\simeq(\Lambda_g)_{\bbC} =V_{\bbC}$. It follows from the construction of the Riemann correspondence for complex abelian varieties that the dual isomorphism
  \[
    \alpha_{\bbC}\colon V^*_{\bbC}\stackrel{\sim}{\longrightarrow}H^1(A,\bbC)\cong H^1_{\mathrm{dR}}(A/\bbC)=\calV_{\xi_{\bbC}}
  \]
  identifies the Hodge decomposition $H^1_{\mathrm{dR}}(A/\bbC)=H^{(1,0)}\oplus H^{(0,1)}$ with the decomposition $V^*_{\bbC}=(V^{(-1,0)})^*\oplus (V^{(0,-1)})^*$ (see e.g. \cite{Mi3}, 6.10., 7.5.), and a direct computation shows that the cocharacter $(\cdot)^{\vee}\circ\nu_h^{-1}$ (with $\nu_h$ as in Def. \ref{CocharDef}) acts on $(V^{(-1,0)})^*$ with weight $1$ and on $(V^{(0,-1)})^*$ with weight $0$, in other words,
  \[
    (\nu_h(z)^{-1})^{\vee}|_{(V^{(-1,0)})^*}=z,\qquad (\nu_h(z)^{-1})^{\vee}|_{(V^{(0,-1)})^*}=1.
  \]
  This means that $(\cdot)^{\vee}\circ\nu_h^{-1}$ induces on $V^*_{\bbC}$ the filtration   
  \[
    V^*_{\bbC}\supset(V^{(-1,0)})^*=\alpha_{\bbC}^{-1}(H^{(1,0)})=\alpha_{\bbC}^{-1}(\Fil^1\calV_{\xi_{\bbC}}),
  \] 
  and further by construction of $s_{\mathrm{dR}}$ the isomorphism $\alpha_{\bbC}$ identifies $s_{\bbC}$ with $s_{\mathrm{dR},\xi_{\bbC}}$. 
  
  Now the isomorphism $\alpha_{\bbC}^{-1}\circ\beta_{\bbC}\colon V^*_{\bbC}\to V^*_{\bbC}$ fixes the tensors $s_{\bbC}$, which means that $\alpha_{\bbC}^{-1}\circ\beta_{\bbC}=g_{\bbC}^{\vee}$ for some $g_{\bbC}\in G(\bbC)$. Note that we have $g_{\bbC}^{\vee}(\beta_{\bbC}^{-1}(\Fil^1\calV_{\xi_{\bbC}}))=\alpha_{\bbC}^{-1}(\Fil^1\calV_{\xi_{\bbC}})$. Conjugating $\lambda_{\bbC}$ with $g_{\bbC}$, we may therefore assume that $(\cdot)^{\vee}\circ\lambda_{\bbC}$ and $(\cdot)^{\vee}\circ\nu_h^{-1}$ both induce the same filtration on $V^*_{\bbC}$. Let $P$ be the stabilizer of this filtration in $G_{\bbC}$, that is, the subgroup of all $\tilde{g}\in G_{\bbC}$ such that $\tilde{g}^{\vee}$ leaves the filtration stable. Then $P\subseteq G_{\bbC}$ is a parabolic subgroup, and both $\lambda_{\bbC}$ and $\nu_h^{-1}$ factor via $P$. Since all maximal tori of $P$ are conjugate over $\bbC$, after conjugation by an element of $P(\bbC)$ we may further assume that both cocharacters factor via the same (automatically split) maximal torus of $P$. But this implies that $\lambda_{\bbC}$ and $\nu_h^{-1}$ also induce (via $(\cdot)^{\vee}$) the same grading on $V^*_{\bbC}$, and hence that they are equal.
\end{proof}

\begin{construction}
  \label{LinConstr}
  Let $k$ be algebraically closed of finite transcendence degree over $\bbF_p$. Let $x\in\scrS(k)$. By Cor. \ref{CrysTensors} the tensors $s_{\mathrm{dR}}^{\circ}$ induce $F$-invariant tensors $s_{\mathrm{cris},x}\subset\bbD_x^{\otimes}$, and we find an isomorphism $\beta\colon (\Lambda^*_{W(k)},s_{W(k)})\stackrel{\sim}{\rightarrow}(\bbD_x,s_{\mathrm{cris},x})$. To $\beta$ we attach an element $g_{\beta}\in G(L(k))$ as follows:
  
   Transporting $F$ via $\beta$, we obtain an injective $\sigma$-linear map $F_{\beta}:=\beta^{-1}\circ F\circ\beta$ on $\Lambda^*_{W(k)}=(\Lambda^*_{\Zp})\otimes_{\Zp}W(k)$. We can write it uniquely as $F_{\beta}=F_{\beta}^{\mathrm{lin}}\circ (1\otimes\sigma)$, where $F_{\beta}^{\mathrm{lin}}$ is by definition an automorphism of $\Lambda^*_{L(k)}$ which fixes the tensors $s_{L(k)}$. Therefore we have $F_{\beta}^{\mathrm{lin}}=g_{\beta}^{\vee}$ for a unique element $g_{\beta}$ of $G(L(k))$.\\
  We can also summarize the construction differently: For each $\beta$ the associated $g_{\beta}\in G(L(k))$ is the unique element such that the diagram
  \[
    \begin{xy}
      \xymatrix{
        \bbD_x^{(\sigma)}\ar[rr]^{F^{\mathrm{lin}}} & & \bbD_x\\
        \Lambda^*_{W(k)}\ar[rr]^{g_{\beta}^{\vee}} \ar[u]^{\beta^{(\sigma)}} & & \Lambda^*_{W(k)}\ar[u]_{\beta}
      }
    \end{xy}
  \]
  commutes (here we make the usual identification $(\Lambda^*_{W(k)})^{(\sigma)}\cong\Lambda^*_{W(k)}$).
\end{construction}

\begin{lemma}
  \label{LinearizationLem}
  Let $k$ be as in Constr. \ref{LinConstr}.
  \begin{enumerate}[(i)]
    \item Let $x\in\scrS(k)$, let $\beta\colon(\Lambda^*_{W(k)},s_{W(k)})\stackrel{\sim}{\longrightarrow}(\bbD_x,s_{\mathrm{cris},x})$ be an isomorphism. Then the isomorphisms between $(\Lambda^*_{W(k)},s_{W(k)})$ and $(\bbD_x,s_{\mathrm{cris},x})$, are exactly the ones of the form $\beta'=\beta\circ h^{\vee}$ for $h\in\calG(W(k))$, further in this case $h$ is uniquely determined and we have $g_{\beta'}=h^{-1}g_{\beta}\sigma(h)$.
    \item For every $x\in\scrS(k)$ and for every isomorphism $\beta\colon (\Lambda^*_{W(k)},s_{W(k)})\stackrel{\sim}{\longrightarrow}(\bbD_x,s_{\mathrm{cris},x})$ we have that $g_{\beta}\in\calG(W(k))\mu(p)\calG(W(k))$.
    
  \end{enumerate}
\end{lemma}

\begin{proof}
  (i) This is clear.
  
  (ii) By the Cartan decomposition for $G(L(k))$ (see \cite{Ti1}, 3.3.3.) we know that $g_{\beta}$ lies in a double coset $\calG(W(k))\eta(p)\calG(W(k))$ for a unique dominant cocharacter $\eta\in X_*(\calT)$. By (i) we may further w.l.o.g. replace $\beta$ by $\beta\circ h^{\vee}$ for a suitable $h\in\calG(W(k))$ to achieve that $g_{\beta}\in\calG(W(k))\eta(p)$. In order to show that $\eta=\mu$, it suffices to check that the base change of $\sigma^{-1}(\eta)$ to $k$ lies in $[\nu^{-1}]$.
  
  Let $\tilde{x}\in\scrS(W(k))$ be a lift of $x$, and identify $\bbD_x\cong\calV_{\tilde{x}}^{\circ}$. Let 
  \[
    \Lambda^*_{W(k)}\supset\beta^{-1}(\Fil^1\calV_{\tilde{x}}^{\circ})
  \]
  be the pullback of the Hodge filtration on $\calV_{\tilde{x}}^{\circ}$. By Prop. \ref{dRExtProp}(iii) there is a cocharacter $\lambda$ of $\calG_{W(k)}$ which induces this filtration, and by Lemma \ref{CocharLem1} we know that $\lambda\in[\nu^{-1}]$.\\
  Reducing the whole situation modulo $p$ we obtain the contravariant Dieudonne space $(\overline{\bbD_x}, \overline{F}, \overline{V})$ associated to the $p$-torsion $\calA_x[p]$ and the isomorphism
  \[
    \bar{\beta}\colon\Lambda^*_k\stackrel{\sim}{\longrightarrow}\overline{\bbD_x}\cong\overline{\calV_{\tilde{x}}^{\circ}}=\calV_x^{\circ}=H_{\mathrm{dR}}^1(\calA_x/k).
  \] 
  By a result of Oda (\cite{Od1}, 5.11.), we have the equality $\Fil^1\calV_x^{\circ}=\ker(\overline{F})$, which implies that 
  \[
    \overline{\beta^{-1}(\Fil^1\calV_{\tilde{x}}^{\circ})}=\bar{\beta}^{-1}(\ker(\overline{F}))=\ker(\bar{\beta}^{-1}\circ\overline{F}\circ\bar{\beta})=\ker(\overline{F_{\beta}}),
  \]
  and this filtration is induced via $(\cdot)^{\vee}$ by the reduction $\bar{\lambda}$ of $\lambda$.\\
  On the other hand, we may write $g_{\beta}=g_0\eta(p)$ for some $g_0\in\calG(W(k))$, therefore
  \[
    F_{\beta}=g_0^{\vee}\circ\eta(p)^{\vee}\circ(1\otimes\sigma)=(1\otimes\sigma)\circ\sigma^{-1}(g_0)^{\vee}\circ\sigma^{-1}(\eta)(p)^{\vee}.
  \]
  Let $\Lambda^*_{W(k)}=\bigoplus_{m\in\bbZ}\Lambda_m^*$ be the grading which is induced by the cocharacter $(\cdot)^{\vee}\circ\sigma^{-1}(\eta)$ on $\Lambda^*_{W(k)}$. The inclusions $p\cdot\Lambda^*_{W(k)}\subseteq\im(F_{\beta})\subseteq\Lambda^*_{W(k)}$ show that we must have $\Lambda^*_m=(0)$ for $m\neq 0,1$, and thus reducing modulo $p$ we find that 
  \[
    \ker(\overline{F_{\beta}})=\ker(\overline{\sigma^{-1}(\eta)(p)^{\vee}})=\overline{\Lambda^*_1}.
  \]
  This implies that the two cocharacters $(\cdot)^{\vee}\circ\overline{\sigma^{-1}(\eta)}$ and $(\cdot)^{\vee}\circ\bar{\lambda}$ induce the same filtration on $\Lambda^*_k$. Now it follows by the same argument as in the proof of Lemma \ref{CocharLem1} that $\overline{\sigma^{-1}(\eta)}$ and $\bar{\lambda}$ are $\calG(k)$-conjugate, which concludes the proof.
\end{proof}

\begin{corollary}
  \label{LinearizationCor}
  Let $k$ be as in Constr. \ref{LinConstr}, let $x,x'\in\scrS(k)$. Let $g_{\beta}, g_{\beta'}\in\calG(W(k))$ be associated to isomorphisms $\beta$ and $\beta'$ between $(\Lambda^*_{W(k)},s_{W(k)})$ and $(\bbD_x,s_{\mathrm{cris},x})$ resp. $(\bbD_{x'},s_{\mathrm{cris},x'})$. Then there is an isomorphism of Dieudonn{\'e} modules $\bbD_x\simeq\bbD_{x'}$ which identifies $s_{\mathrm{cris},x}$ with $s_{\mathrm{cris},x'}$ if and only if $g_{\beta'}=hg_{\beta}\sigma(h)^{-1}$ for some $h\in\calG(W(k))$.
\end{corollary}

\begin{proof}
  By Constr. \ref{LinConstr} the existence of an isomorphism $\bbD_x\simeq\bbD_{x'}$ which respects the tensors on both sides is equivalent to the existence of an automorphism $\delta$ of $(\Lambda^*_{W(k)},s_{W(k)})$ such that 
  \[
    g_{\beta'}^{\vee}\circ(1\otimes\sigma)\circ\delta=\delta\circ g_{\beta}^{\vee}\circ(1\otimes\sigma)\tag{*}
  \]
  Every such automorphism must be of the form $\delta=h^{\vee}$ for a unique $h\in\calG(W(k))$, and an easy calculation shows that the property (*) is equivalent to $g_{\beta'}=hg_{\beta}\sigma(h)^{-1}$. 
\end{proof}

\subsection{The Newton stratification}
\label{NewtStratSec}

  In order to define the Newton stratification on $\scrS\otimes\kappa(v)$ we recall some facts on $\sigma$-conjugacy classes: 
  
  Let $k$ be algebraically closed of characteristic $p$. We denote by $[g]$ the $\sigma$-conjugacy class of an element $g\in G(L(k))$, and by $B(G)$ the set of all $\sigma$-conjugacy classes in $G(L(k))$. This definition is in fact independent of $k$ in the following sense: If $k'$ is any algebraically closed field of characteristic $p$, then every inclusion $k\subseteq k'$ induces a bijection between the $\sigma$-conjugacy classes of $G(L(k))$ and those of $G(L(k'))$ (see \cite{RR}, 1.3.).
  
  To describe the set $B(G)$ one uses the two maps
  \[
    \nu_G\colon B(G)\longrightarrow \left(W\setminus X_*(\calT)_{\bbQ}\right)^{\langle\sigma\rangle},\qquad \kappa_G\colon B(G)\longrightarrow\pi_1(G)_{\langle\sigma\rangle},
  \]
  usually called the Newton map and the Kottwitz map of $G$ (see \cite{Ko1}, \cite{RR}). As explained in (\cite{RR}, \S2), the set $(W\setminus X_*(\calT)_{\bbQ})^{\langle\sigma\rangle}$ is endowed with a partial order $\preceq$ which generalizes the "lying above" order for Newton polygons. 
  
  We define 
  \[
    \bar{\mu}:=\frac{1}{r}\sum_{i=0}^{r-1}\sigma^i(\mu)\in (X_*(\calT)_{\bbQ})^{\langle\sigma\rangle},
  \]
  where $r$ is some integer such that $\sigma^r(\mu)=\mu$ (obviously this does not depend on the choice of $r$), and also identify $\bar{\mu}$ with its image in $(W\setminus X_*(\calT)_{\bbQ})^{\langle\sigma\rangle}$. Let $\mu^{\natural}\in\pi_1(G)_{\langle\sigma\rangle}$ be the image of $\mu$ under the natural projection  
  \[
    X_*(\calT)\to\pi_1(G)_{\langle\sigma\rangle}=(X_*(\calT)/\langle \alpha^{\vee}\mid\alpha^{\vee}\in\Phi^{\vee}\rangle)_{\langle\sigma\rangle}.
  \]

  \begin{definition}
    Let $B(G,\mu):=\big\{b\in B(G)\mid \kappa_G(b)=\mu^{\natural},\ \nu_G(b)\preceq\bar{\mu}\big\}$. We endow $B(G,\mu)\subset B(G)$ with the induced partial order $\preceq$.
  \end{definition}
  
  By work of Kottwitz-Rapoport (\cite{KR}), Lucarelli (\cite{Lu}) and Gashi (\cite{Ga}) we know that $B(G,\mu)$ is exactly the image of the double coset $\calG(W(k))\mu(p)\calG(W(k))$ in $B(G)$ (for any algebraically closed field $k$ of char. $p$). We summarize some combinatorial properties of this set in the following remark:
    
\begin{remark}
  \label{NewtParRem}
  \begin{enumerate}[(1)]
    \item $B(G,\mu)$ is a finite set (see \cite{RR}, 2.4.).   
    \item The set $B(G,\mu)$ contains a unique maximal element $b_{\mathrm{max}}$ with respect to $\preceq$, namely the $\sigma$-conjugacy class $[\mu(p)]$: In fact, the characterization of the Newton map in (\cite{Ko1}, 4.3.) shows that $\bar{\mu}$ is nothing but the image of $[\mu(p)]$ under $\nu_G$. On the other hand, it follows directly from the definition of $\kappa_G$ that $\mu^{\natural}$ is the image of $[\mu(p)]$ under $\kappa_G$. Therefore we have $[\mu(p)]\in B(G,\mu)$, and clearly the inequality $b\preceq [\mu(p)]$ holds for all $b\in B(G,\mu)$.
    \item $B(G,\mu)$ contains a unique basic element $b_{\mathrm{bas}}$ (it corresponds to $\mu^{\natural}$ under the bijection between basic $\sigma$-conjugacy classes and $\pi_1(G)_{\langle\sigma\rangle}$, see \cite{RR}, 1.15.), which is also the unique minimal element with respect to $\preceq$ in $B(G,\mu)$.  
  \end{enumerate}
\end{remark}
  
  Let us now define the Newton stratification on $\scrS\otimes\kappa(v)$: Consider a point $x\in\scrS\otimes\kappa(v)$, let $k(x)$ be the residue class field of $\scrS$ in $x$. Let $k$ be some algebraic closure of $k(x)$, and let $\hat{x}$ be the associated geometric point. Constr. \ref{LinConstr} associates to each isomorphism $\beta\colon (\Lambda^*_{W(k)},s_{W(k)})\simeq(\bbD_{\hat{x}},s_{\mathrm{cris},\hat{x}})$ an element $g_{\beta}\in G(L(k))$, and Lemma \ref{LinearizationLem} shows that the $\sigma$-conjugacy class $[g_{\beta}]$ is independent of the choice of $\beta$ and lies in $B(G,\mu)$. Further, this element only depends on $x$ and not on the choice of the algebraic closure of $k(x)$ in the sense explained at the beginning of this subsection. Thus the assignment $x\mapsto[g_{\beta}]$ gives a well-defined map
\[
  \mathrm{Newt}\colon\scrS\otimes\kappa(v)\longrightarrow B(G,\mu).
\]

\begin{definition}
  \label{NewtonStratDef}
  \begin{enumerate}[(i)]
    \item For an element $b\in B(G,\mu)$ we set $\calN^b:=\mathrm{Newt}^{-1}(\{b\})\subseteq\scrS\otimes\kappa(v)$. We call $\calN^b$ the \emph{Newton stratum} of $b$.
    \item We call the stratum $\calN^{b_{\mathrm{max}}}$ the $\mu$-\emph{ordinary locus} in $\scrS\otimes\kappa(v)$. 
  \end{enumerate}
\end{definition}

A priori, the $\calN^b$ are just subsets of $\scrS_0$, but we will see below that they are in fact locally closed, which justifies the name "strata".

\begin{remark}
  \begin{enumerate}[(1)]
    \item In view of Cor. \ref{LinearizationCor} we see that two points $x_1,x_2\in\scrS\otimes\kappa(v)$ lie in the same Newton stratum if and only if the following holds: If $k$ is any algebraically closed field such that $k(x_1)$ and $k(x_2)$ both embed into $k$, with associated points $\hat{x}_1,\hat{x}_2\in\scrS(k)$, then there is an isomorphism of isocrystals $(\bbD_{\hat{x}_1})_{\bbQ}\simeq(\bbD_{\hat{x}_2})_{\bbQ}$ which identifies the tensors $s_{\mathrm{cris},\hat{x}_1}$ with $s_{\mathrm{cris},\hat{x}_2}$. 
    \item In the case of a PEL-type Shimura datum, at each geometric point $\hat{x}$ of $\scrS\otimes\kappa(v)$ the tensors $s_{\mathrm{cris},\hat{x}}$ describe the additonal structure on $\bbD_{\hat{x}}$ (cf. Rem. \ref{PELRem2}). Hence in this case the Newton strata from Def. \ref{NewtonStratDef} agree with those considered in \cite{RR}.
  \end{enumerate}
\end{remark}

It is natural to conjecture that the Grothendieck specialization theorem holds for the $\calN^b$. That is, if $x_1,x_2\in\scrS\otimes\kappa(v)$ such that $x_2$ is a specialization of $x_1$, then we expect that $\mathrm{Newt}(x_2)\preceq\mathrm{Newt}(x_1)$. To our knowledge, this has not yet been established for a general Shimura variety of Hodge type. In the PEL-case, it follows from the fact that the isocrystal over $\scrS\otimes\kappa(v)$ associated to the $p$-divisible group $\calA\otimes\kappa(v)[p^{\infty}]$, with induced additional structure, can be understood as an isocrystal with $G$-strucure in the sense of (\cite{RR}, \S3). 

There is, however, the following result of Vasiu. Since $\calA\otimes\kappa(v)$ is a polarized abelian scheme over $\scrS\otimes\kappa(v)$, we have the classical stratification of $\scrS\otimes\kappa(v)$ by Newton polygons, as defined by Oort. For a symmetric Newton polygon $\Delta\in B(\GSp(V))$, denote the corresponding stratum by $\calN_{\mathrm{NP}}^{\Delta}$. Then every $\calN^b$ lies in a unique stratum $\calN_{\mathrm{NP}}^{\Delta(b)}$, this defines a map $B(G,\mu)\to B(\GSp(V)),\ b\mapsto\Delta(b)$, which should be thought of as "forgetting the tensor structure".

\begin{proposition}[\cite{Va2}, 5.3.1.(ii)]
  \label{NewtonStratProp}
  For every $b\in B(G,\mu)$ the stratum $\calN^b$ is an open and closed subset of  $\calN_{\mathrm{NP}}^{\Delta(b)}$. 
\end{proposition}

As a consequence, since the strata $\calN_{\mathrm{NP}}^{\Delta}$ are locally closed subsets of $\scrS\otimes\kappa(v)$, the same holds true for the strata $\calN^b$. We endow them with the structure of a reduced subscheme of $\scrS\otimes\kappa(v)$.

\subsection{The Ekedahl-Oort stratification}
\label{EOStratSec}

Recall that $\Fbar$ denotes a fixed algebraic closure of $\Fp$. We now describe the Ekedahl-Oort stratification on $\scrS\otimes\Fbar$ which has been constructed and studied by Zhang in \cite{Zh1}. However, we give a slightly different definition, as we feel that one should work with $\Lambda^*$ rather than with $\Lambda$, further we make the definition independent of the choice of a cocharacter. The main results of \cite{Zh1} remain true with the obvious changes.\\ 

The definition of the Ekedahl-Oort stratification is based on the theory of $\calG_{\Fp}$-Zips which has been developed in \cite{PWZ1}, \cite{PWZ2}, and which we will apply to a cocharacter in the conjugacy class $[\nu^{-1}]$. The stratification is parametrized by the subset $^JW$ of the Weyl group $(W,S)$, where $J\subseteq S$ is the type of the conjugacy class $[\nu^{-1}]$. Let us recall how this data is defined:  We view $[\nu^{-1}]$ as an element of $W\setminus X_*(\calT)$, as explained at the beginning of this section. Let $\chi_{\mathrm{dom}}\in X_*(\calT)$ be the unique dominant element lying in $[\nu^{-1}]$, then $J:=\{s\in S\mid s(\chi_{\mathrm{dom}})=\chi_{\mathrm{dom}}\}$. The subset $J$ generates a subgroup $W_J$ of $W$, and $(W_J,J)$ is again a coxeter group. Every left coset of $W_J$ in $W$ contains a unique minimal element with respect to the length function on $(W,S)$, and the set $^JW$ set of these elements forms a full set of representatives for the left cosets of $W_J$ in $W$ (see for example \cite{BB}, \S 2.4.).  
  
  Let $w_0$ and $w_{0,J}$ be the longest elements of $W$ and $W_J$ respectively and let $x_J:=w_0w_{0,J}$. We have a partial order $\preceq$ on ${^JW}$ (\cite{PWZ1}, 6.3.) given as
  \[
    w'\preceq w\ :\Longleftrightarrow\ yw'\sigma(x_Jyx_J^{-1})\leq w\ \text{ for some }y\in W_J.
  \]
  This partial order induces a topology on $^JW$ such that a subset $U\subset{^JW}$ is open if and only if for any $w'\in U$ and any $w\in{^JW}$ with $w'\preceq w$ one also has $w\in U$.\\

  If $X$ is a scheme or a sheaf over an $\Fp$-scheme $S$, denote by $X^{(\sigma)}$ its pullback by the absolute Frobenius $x\mapsto x^p$ on $S$, and likewise for morphisms of objects over $S$. Let $\kappa$ be an algebraic extension of $\Fp$. Since $\calG_{\kappa}=\calG\otimes\kappa$ is defined over $\Fp$, we have $\calG_{\kappa}^{(\sigma)}\cong\calG_{\kappa}$ canonically (compare Section \ref{sigmaPrep}). In particular, for every subgroup $H\subseteq \calG_{\kappa}$ the pullback $H^{(\sigma)}$ is again a subgroup of $\calG_{\kappa}$. Further, the composition $\calG_{\kappa}\stackrel{\mathrm{Frob}_p}{\longrightarrow}\calG_{\kappa}^{(\sigma)}\cong\calG_{\kappa}$ of the relative Frobenius morphism of $\calG_{\kappa}$ with this canonical isomorphism is an isogeny of the algebraic group $\calG_{\kappa}$. By abuse of notation we denote this isogeny again by $\sigma$. 
  
   Let $\chi$ be a cocharacter of $\calG_{\kappa}$ such that $\chi_{\Fbar}\in[\nu^{-1}]$. Using the identification $\calG_{\kappa}^{(\sigma)}\cong\calG_{\kappa}$ we can view $\chi^{(\sigma)}$ as a cocharacter of $\calG_{\kappa}$ as well. Let $P_+, P_-\subseteq\calG_{\kappa}$ be the parabolic subgroups which are characerized by the property that $\Lie(P_{\pm})$ is the sum of the non-negative (non-positive) weight spaces with respect to the adjoint operation of $\chi$ on $\Lie(\calG_{\kappa})$. Denote by $U_+$ and $U_-$ the corresponding unipotent radicals and by $M$ the common Levi subgroup of $P_+$ and $P_-$.\\

\begin{definition}[\cite{PWZ2}, 3.1.]
  \label{GZipDef}
  Let $S$ be a scheme over $\kappa$. A $\calG_{\Fp}$\emph{-zip} of type $\chi$ over $S$ is a quadruple $\underline{I}=(I,I_+,I_-,\iota)$, where
  \begin{enumerate}[(1)]
    \item $I$ is a right $\calG_{\kappa}$-torsor for the fpqc-topology on $S$,
    \item $I_+\subseteq I$ and $I_-\subseteq I$ are subsheaves such that $I_+$ is a $P_+$-torsor and $I_-$ is a $P_-^{(\sigma)}$-torsor,
    \item $\iota\colon I_+^{(\sigma)}/U_+^{(\sigma)}\stackrel{\sim}{\longrightarrow}I_-/U_-^{(\sigma)}$ is an isomorphism of $M^{(\sigma)}$-torsors.
  \end{enumerate}
  A morphism $\underline{I}\longrightarrow\underline{I}'$ of $\calG_{\Fp}$-zips over $S$ is a $\calG_{\kappa}$-equivariant map $I\to I'$ which maps $I_+$ to $I_+'$ and $I_-$ to $I_-'$ and is compatible with $\iota$ and $\iota'$.
\end{definition}

With the natural notion of pullback the $\calG_{\Fp}$-zips of type $\chi$ form a stack $\calG_{\Fp}\mathtt{-Zip}_{\kappa}^{\chi}$ over $(\mathbf{Sch}/\kappa)$ (\cite{PWZ2}, 3.2.).

\begin{proposition}[\cite{PWZ2}, 3.12., 3.20., 3.21.]
  \label{StackTopSpace}
  $\calG_{\Fp}\mathtt{-Zip}_{\kappa}^{\chi}$ is a smooth algebraic stack of dimension $0$ over $\kappa$, and there is a homeomorphism of topological spaces
  \[
    \calG_{\Fp}\mathtt{-Zip}_{\kappa}^{\chi}(\Fbar)\simeq{^J}W
  \]
  where $^JW$ is endowed with the topology given by $\preceq$. 
\end{proposition}

So in particular, there is a bijection between the set of isomorphism classes of $\calG_{\Fp}$-zips of type $\chi$ over $\Fbar$ and the set $^JW$. We will give a precise description of this bijection in the following section.

\begin{construction}
  \label{EOConstr}
  Let $\overline{\calV^{\circ}}:=H_{\mathrm{dR}}^1(\calA\otimes\kappa(v)/\scrS\otimes\kappa(v))$, which is the reduction mod $p$ of $\calV^{\circ}$. Let $\calC:=\overline{\Fil^1\calV^{\circ}}=\Fil^1\overline{\calV^{\circ}}$ be the Hodge filtration, this is a locally direct summand of $\overline{\calV^{\circ}}$. As explained in (\cite{MW}, \S7), the conjugate Hodge spectral sequence also gives rise to a locally direct summand $\calD:=R^1\pi_*(\scrH^0(\Omega^{\bullet}_{\calA\otimes\kappa(v)/\scrS\otimes\kappa(v)}))$ of $\overline{\calV^{\circ}}$, and the (inverse) Cartier homomorphism provides isomorphisms
  \[
    \phi_0\colon (\overline{\calV^{\circ}}/\calC)^{(\sigma)}\stackrel{\sim}{\longrightarrow}\calD,\qquad \phi_1\colon \calC^{(\sigma)}\stackrel{\sim}{\longrightarrow}\overline{\calV^{\circ}}/\calD.
  \]
  
  We now fix a cocharacter $\chi$ and a \emph{finite} extension $\kappa$ of $\kappa(v)$ such that $\chi$ is defined over $\kappa$ and such that $\chi_{\Fbar}\in[\nu^{-1}]$. Recall that $\calG_{\kappa}$ and thus $\chi$ and $\chi^{(\sigma)}$ act on $\Lambda^*_{\kappa}$ via the contragredient representation $(\cdot)^{\vee}$. Let 
  \[
    \Lambda^*_{\kappa}=\Fil^0_{\chi}\supset \Fil^1_{\chi}\supset(0),\qquad (0)\subset\Fil_0^{\chi^{(\sigma)}}\subset\Fil_1^{\chi^{(\sigma)}}=\Lambda^*_{\kappa}
  \]
  be the descending resp. ascending filtration given in this way by $\chi$ and by $\chi^{(\sigma)}$. Then $P_+$ is nothing but the stabilizer of $\Fil^{\bullet}_{\chi}$ in $\calG_{\kappa}$, that is,
  \[
    P_+=\{g\in\calG_{\kappa}\mid g^{\vee}(\Fil^1_{\chi})=\Fil^1_{\chi}\},
  \]
  and in the same fashion $P_-^{(\sigma)}$ is the stabilizer of $\Fil_{\bullet}^{\chi^{(\sigma)}}$.

  We denote by $\bar{s}_{\mathrm{dR}}$ the reduction of the tensors $s_{\mathrm{dR}}^{\circ}$ to $\overline{\calV^{\circ}}$, and by $\bar{s}$ the base change of $s\subset(\Lambda^*_{\Zp})^{\otimes}$ to $\Lambda^*_{\kappa}$. Define
\begin{align*}
  I:= & \mathbf{Isom}_{\scrS\otimes\kappa}\big((\Lambda^*_{\kappa},\bar{s})\otimes\calO_{\scrS\otimes\kappa},\, (\overline{\calV^{\circ}},\bar{s}_{\mathrm{dR}})\otimes\calO_{\scrS\otimes\kappa}\big),\\
  I_+:= & \mathbf{Isom}_{\scrS\otimes\kappa}\big((\Lambda^*_{\kappa},\bar{s},\Fil^{\bullet}_{\chi})\otimes\calO_{\scrS\otimes\kappa},\, (\overline{\calV^{\circ}},\bar{s}_{\mathrm{dR}},\overline{\calV^{\circ}}\supset\calC)\otimes\calO_{\scrS\otimes\kappa}\big),\\
  I_-:= & \mathbf{Isom}_{\scrS\otimes\kappa}\big((\Lambda^*_{\kappa},\bar{s},\Fil_{\bullet}^{\chi^{(\sigma)}})\otimes\calO_{\scrS\otimes\kappa},\, (\overline{\calV^{\circ}},\bar{s}_{\mathrm{dR}},\calD\subset\overline{\calV^{\circ}})\otimes\calO_{\scrS\otimes\kappa}\big).
\end{align*}
  We have a natural right action of $\calG_{\kappa}$ on $I$ given by $\beta\cdot g:=\beta\circ g^{\vee}$, and $I_+$ and $I_-$ inherit actions of $P_+$ and $P_-^{(\sigma)}$. 
\end{construction}

\begin{proposition}[\cite{Zh1}, 2.4.1.]
  \label{GzipProp}
  The Cartier isomorphisms induce an isomorphism $\iota\colon I_+^{(\sigma)}/U_+^{(\sigma)}\stackrel{\sim}{\longrightarrow}I_-/U_-^{(\sigma)}$ such that the tuple $\underline{I}=(I,I_+,I_-,\iota)$ is a $\calG_{\Fp}$-zip of type $\chi$ over $\scrS\otimes\kappa$.
\end{proposition}

\begin{proof}
  Let us show for our definitions that for every closed point $x$ of $\scrS\otimes\kappa$ the fibers $I_x$ and $(I_+)_x$ are trivial torsors for $\calG_{\kappa}$ resp. $P_+$: Let $k(x)$ be the residue class field of $\scrS$ at $x$, which is a finite extension of $\kappa(v)$, let $\tilde{x}\in\scrS(W(k(x)))$ be a lift of $x$. By Prop. \ref{dRExtProp}(iii) and Lemma \ref{CocharLem1} we find an isomorphism $\bar{\beta}\colon(\Lambda^*_{k(x)},s_{k(x)})\stackrel{\sim}{\longrightarrow}(\overline{\calV^{\circ}}_x, \bar{s}_{\mathrm{dR},x})$ and a cocharacter $\bar{\lambda}$ of $\calG_{k(x)}$ which induces the filtration $\Lambda^*_{k(x)}\supset\bar{\beta}^{-1}(\calC_x)$, and further we have $\bar{\lambda}_{\Fbar}\in[\nu^{-1}]$. Then $\bar{\beta}$ lies in $I_x(k(x))$, which shows that $I_x$ is trivial. Moreover, $\bar{\lambda}$ is conjugate to $\chi$ over some finite extension of $k(x)$. This implies that $(I_+)_x$ is a $P_+$-torsor, and as $P_+$ is connected and $k(x)$ is finite this torsor must be trivial.
  
  Now all the arguments in the sections 2.2. - 2.4. and in the proof of 2.4.1. of \cite{Zh1} carry over to our definition of $I$, $I_+$ and $I_-$ with the necessary adjustments. 
  
\end{proof}  
  
For every scheme $S$ over $\scrS\otimes\kappa$ one now obtains a $\calG_{\Fp}$-zip over $S$ by pulling back the $\calG_{\Fp}$-zip $\underline{I}$, in other words, $\underline{I}$ defines a morphism of algebraic stacks
\[
  \zeta\colon\scrS\otimes\kappa\to\calG_{\Fp}\mathtt{-Zip}_{\kappa}^{\chi}.
\]

\begin{theorem}[\cite{Zh1}, 3.1.2.]
  \label{ZhangThm}
  The morphism $\zeta$ is smooth. In particular it induces a continuous and open map of topological spaces 
  \[  
    \zeta(\Fbar)\colon\scrS(\Fbar)\longrightarrow \calG_{\Fp}\mathtt{-Zip}_{\kappa}^{\chi}(\Fbar)\simeq{^JW}.
  \]
\end{theorem}

\begin{proof}
  Again, the proof of (\cite{Zh1}, 3.1.2.) goes through with the obvious changes.
\end{proof}
  
\begin{remark}
  \label{EOIndepRem}
  Though the definition of $\zeta$ depends on the choice of a cocharacter $\chi$, the resulting map $\zeta(\Fbar)\colon\scrS(\Fbar)\to{^JW}$ is in fact independent of $\chi$. This is a consequence of the following two observations:
  \begin{enumerate}[(1)]
    \item Let $\kappa'$ be a finite extension of $\kappa$, let $\chi'=\chi_{\kappa'}$, and let $\underline{I}'$ be the $\calG_{\Fp}$-zip of type $\chi'$ over $\scrS\otimes\kappa'$ given by Constr. \ref{EOConstr}. Then we have $\calG_{\Fp}\mathtt{-Zip}_{\kappa'}^{\chi'}=\calG_{\Fp}\mathtt{-Zip}_{\kappa}^{\chi}\otimes\kappa'$ and the equality $\underline{I}'=\underline{I}\otimes\kappa'$, which means that $\chi$ and $\chi'$ induce the same map $\zeta(\Fbar)$.
    \item Let $\chi'$ be a cocharacter of $\calG_{\kappa}$ which is conjugate to $\chi$ over $\kappa$, say $\chi'=\mathrm{int}(g)\circ\chi$ for some $g\in\calG(\kappa)$. Let $P'_{\pm}\subseteq\calG_{\kappa}$ be the associated parabolic subgroups, with common Levi subgroup $M'$, and again denote by $\underline{I}'$ the $\calG_{\Fp}$-zip associated to $\chi'$ (over $\kappa$). Applying the propositions \ref{GzipProp} and \ref{StackTopSpace} to $\kappa$ and $\chi'$ one obtains a map
    \[  
      \zeta'(\Fbar)\colon\scrS(\Fbar)\to\calG_{\Fp}\mathtt{-Zip}_{\kappa'}^{\chi'}(\Fbar)\simeq{^JW}.
    \]
    As $P'_{\pm}=g(P_{\pm})g^{-1}$ and $M'=gMg^{-1}$, the element $g$ defines an isomorphism of algebraic stacks
    \begin{align*}
      \Xi\colon\calG_{\Fp}\mathtt{-Zip}_{\kappa}^{\chi} & \stackrel{\sim}{\longrightarrow}\calG_{\Fp}\mathtt{-Zip}_{\kappa}^{\chi'}\\
      \big(I,I_+,I_-,\iota\big) & \longmapsto\big(I,\;(I_+)\cdot g^{-1},\;(I_-)\cdot \sigma(g)^{-1},\; r_{\sigma(g){-1}}\circ\iota\circ r_{\sigma(g)}\big).
    \end{align*}
    (Here $r_{\sigma(g)}$ and $r_{\sigma(g)^{-1}}$ are the obvious isomorphisms $(I'_+)^{(\sigma)}/(U'_+)^{(\sigma)}\simeq I_+^{(\sigma)}/U_+^{(\sigma)}$ and $I_-/U_-^{(\sigma)}\simeq I'_-/(U'_-)^{(\sigma)}$ given by multiplication with $\sigma(g)$ resp. $\sigma(g)^{-1}$ on the right.)\\
    It is easy to see that $\Xi(\underline{I})=\underline{I}'$. Further, going through the classification of $\calG_{\Fp}$-zips in \cite{PWZ1}, \cite{PWZ2} (see also \ref{ZipClass}), a straightforward but tedious computation shows that $\Xi$ is compatible with the homeomorphisms from Prop. \ref{StackTopSpace}, which implies that $\zeta'(\Fbar)=\zeta(\Fbar)$.
  \end{enumerate}
\end{remark}

Due to Theorem \ref{ZhangThm} and the definition of the topology on $^JW$, the inverse images of elements $w\in{^J}W$ under $\zeta(\Fbar)$ are the $\Fbar$-valued points of locally closed subsets $\calS^w\subseteq\scrS\otimes\Fbar$.

\begin{definition}
  \label{EODef}
  For $w\in{^JW}$ we call $\calS^w\subseteq\scrS\otimes\Fbar$ the \emph{Ekedahl-Oort stratum} associated to $w$. We endow the strata $\calS^w$ with the reduced subscheme structure.
\end{definition}

Let us collect some information on these strata:

\begin{remark}
  \label{EOPropRem}
  \begin{enumerate}[(1)]
    \item Each $\calS^w$ is either empty or equidimensional of dimension $l(w)$ (see \cite{Zh1}, 3.1.6.).
    \item The $\calS^w$ form a stratification of $\scrS\otimes\Fbar$ in the strict sense: For every $w\in{^JW}$ we have
    \[
      \overline{\calS^w}=\bigcup_{w'\preceq w}\calS^w.
    \]
    This follows from the fact that $\zeta$ is an open map and the structure of the topological space $^JW$.
    \item The set $^JW$ contains a unique maximal element with respect to $\preceq$, namely $w_{\mathrm{max}}:=w_{0,J}w_0$, and a unique minimal element  $w_{\mathrm{min}}:=1$. By (2), $\calS^{w_{\mathrm{max}}}$ is the unique open EO-stratum and is dense in $\scrS\otimes\Fbar$, and $\calS^{w_{\mathrm{min}}}$ is closed and contained in the closure of each stratum $\calS^w$.
    \item We do not know whether all Ekedahl-Oort strata are nonempty. In view of (2) and (3) this is equivalent to the question whether $\calS^{w_{\mathrm{min}}}$ is nonempty. 
  \end{enumerate}
\end{remark}

\section{Comparing the stratifications}
\label{StratCompSec}

We will now restrict our attention to the geometric fiber $\scrS\otimes\Fbar$, where the Newton stratification and the Ekedahl-Oort stratification are both defined. The question as to how these stratifications are related to each other can be studied by looking at their $\Fbar$-valued points, since all strata are locally closed subvarieties of $\scrS\otimes\Fbar$. Let us fix some new notations:

\begin{notation}
  We still denote by $\Fbar$ a fixed algebraic closure of $\Fp$. Let $\calO:=W(\Fbar)$ and $L:=L(\Fbar)$. We write $\calS:=\scrS(\Fbar)$. By abuse of notation we will frequently identify geometric objects over $\Fbar$ with their $\Fbar$-valued points. For example we denote the $\Fbar$-valued points of $\calN^b$ resp. of $\calS^w$ by the same symbols. With these notations we have the decompositions
\[
  \calS=\bigcup_{b\in B(G,\mu)}^{\circ}\calN^b,\qquad \calS=\bigcup_{w\in{^JW}}^{\circ}\calS^w.
\]
We write again $(\calB,\calT)$ for the base change to $\calO$ of our fixed Borel pair, and denote by $(B,T)$ its base change to $\Fbar$. For every $w\in W$ we choose a representative $\dot{w}\in N_{\calG}(\calT)(\calO)$.
  
  Let $K:=\calG(\calO)\subseteq G(L)$. The projection $\calO\twoheadrightarrow\calO/(p)=\Fbar$ induces a surjective homomorphism $K\to\calG(\Fbar),\ g\mapsto\bar{g}$. For any subgroup $H\subseteq K$ we will denote its image in $\calG(\Fbar)$ by $\overline{H}$. The Frobenius $\sigma$ acts on $K$ and on $\calG(\Fbar)$, and these operations are compatible in the sense that $\sigma(\bar{g})=\overline{\sigma(g)}$. Let $K_1:=\{g\in K\mid \bar{g}=1\}$, this is a normal subgroup of $K$.  
\end{notation}

\begin{definition}
\label{ConjClNotDef}
  \begin{enumerate}[(i)]
    \item We write $\langle g\rangle:=\{hg\sigma(h)^{-1}\mid h\in K\}$ for the $K$-$\sigma$-conjugacy class of an element $g\in G(L)$.
    \item We set $\calC(\calG,\mu):=\{\langle g\rangle\mid g\in K\mu(p)K\}$.
  \end{enumerate}
\end{definition}

\subsection{A factorization lemma}
\label{FactorSubsec}

It follows from Constr. \ref{LinConstr} and Lemma \ref{LinearizationLem} that there is a well-defined map 
\[
  \gamma\colon\calS\longrightarrow \calC(\calG,\mu)
\]
which is given as follows: If $x\in\calS$ and $\beta\colon(\Lambda_{\calO}^*,s_{\calO})\simeq(\bbD_x,s_{\mathrm{cris},x})$ is an isomorphism as in Cor. \ref{CrysTensors}, then $\gamma(x):=\langle g_{\beta}\rangle$, where $g_{\beta}\in G(L)$ such that $\beta^{-1}\circ F\circ\beta=:F_{\beta}=g_{\beta}^{\vee}\circ(1\otimes\sigma)$, here as usual $(\cdot)^{\vee}$ is the contragredient representation.

\begin{remark}
\label{gammaRem}
  In the case of a PEL-type Shimura variety the map $\gamma$ can be shown to be \emph{surjective} (\cite{VW}, Thm. 11.2.). We do not know whether the surjectivity of $\gamma$ holds in general, though we expect this to be true.
\end{remark}

Let 
\[
  \tilde{\theta}\colon\calC(\calG,\mu)\longrightarrow B(G,\mu),\quad\langle g\rangle\longmapsto[g]
\]
be the natural map and let $\theta:=\tilde{\theta}\circ\gamma\colon\calS\to B(G,\mu)$. Then by Def. \ref{NewtonStratDef} we have $\calN^b=\theta^{-1}(\{b\})$ for each $b\in B(G,\mu)$, and further we can describe the fibers of $\tilde{\theta}$ as follows:
\[
  \tilde{\theta}^{-1}(\{b\})=\{\langle g\rangle\mid g\in K\mu(p)K\cap b\}=:\calC(\calG,\mu)\cap b,\quad b\in B(G,\mu).
\]

On the other hand for every $w\in{^JW}$ the associated Ekedahl-Oort stratum is by definition given as $\calS^w=\zeta^{-1}(\{w\})$. Here we simply write $\zeta$ for the map $\zeta(\Fbar)\colon \calS\to{^JW}$ from Theorem \ref{ZhangThm}. We will now explain that one also has a factorization $\tilde{\zeta}\colon\calC(\calG,\mu)\to{^JW}$ such that $\zeta=\tilde{\zeta}\circ\gamma$, and give a precise description of the inverse image $\tilde{\zeta}^{-1}(\{w\})$ for $w\in{^JW}$.\\

As we have seen in Remark \ref{EOIndepRem}, the map $\zeta$ does not depend on the choice of the cocharacter $\chi$, nor on the choice of $\kappa$. We may and will therefore suppose without loss of generality that $\chi$ is the unique dominant cocharacter contained in $[\nu^{-1}]$, in other words that $\chi=\sigma^{-1}(\mu)$. To $\chi$ we have the associated subgroups $P_{\pm}$, $U_{\pm}$ and $M$ of $\calG_{\Fbar}$ as in Section \ref{EOStratSec}. By (\cite{SGA3}, Exp. XXVI) we may extend these groups to $\calG_{\calO}$ as follows: Let $\calM\subseteq\calG_{\calO}$ be the centralizer of $\chi$ in $\calG_{\calO}$. Let $\calP_+$ be the (unique) parabolic subgroup of $\calG_{\calO}$ with Levi subgroup $\calM$ which contains $\calB$, let $\calP_-$ be its opposite parabolic, and denote their unipotent radicals by $\calU_{\pm}$. Then we have for example $P_+=\overline{\calP_+(\calO)}$ and $P_-^{(\sigma)}=\overline{\sigma(\calP_-(\calO))}$.

\begin{blank}
\label{ZipClass}
  Let us review in detail the homeomorphism $\calG_{\Fp}\mathtt{-Zip}^{\chi}(\Fbar)\simeq{^JW}$ from Prop. \ref{StackTopSpace} in this situation: The Frobenius isogeny gives a morphism
  \[
    \sigma\colon P_+/U_+\cong M\longrightarrow M^{(\sigma)}\cong P_-^{(\sigma)}/U_-^{(\sigma)},
  \]
  such that the tuple $(\calG_{\Fbar},P_+, P_-^{(\sigma)},\sigma)$ is an algebraic zip datum in the sense of (\cite{PWZ1}, 3.1.). The associated zip group is defined as
  \[
    E_{\chi}:=\big\{(mu_+,\sigma(m)u_-)\mid u_+\in U_+, u_-\in U_-^{(\sigma)}, m\in M\big\}\subseteq P_+\times P_-^{(\sigma)}.
  \]
  It acts on $\calG_{\Fbar}$ on the left via $(p_+,p_-)\cdot g=p_+gp_-^{-1}$. 
  
  The isomorphism classes of $\calG_{\Fp}$-zips of type $\chi$ over $\Fbar$ are identified with the orbits under this action by the following construction: Every $a\in\calG_(\Fbar)$ defines a $\calG_{\Fp}$-zip $\underline{I}_a$ over $\Fbar$ by setting
  \[
    \underline{I}_a:=(\calG_{\Fbar},\, P_+,\, a\cdot P_-^{(\sigma)},\, \iota_a),
  \]    
  where $\iota_a$ is given by multiplication with $a$ on the left, more precisely, 
  \[
    \iota_a\colon P_+^{(\sigma)}/U_+^{(\sigma)}\cong M^{(\sigma)}\stackrel{a\cdot}\longrightarrow a\cdot M^{(\sigma)}\cong a\cdot P_-^{(\sigma)}/U_-^{(\sigma)}.
  \]
  A zip of this form is called a \emph{standard} $\calG_{\Fp}$-zip over $\Fbar$. By (\cite{PWZ2}, 3.5.), every $\calG_{\Fp}$-zip $\underline{I}$ of type $\chi$ over $\Fbar$ is isomorphic to a standard zip: If $i_+\in I_+$ and $i_-\in I_-$ such that $\iota(i_+^{(\sigma)}U_+^{(\sigma)})=i_-U_-^{(\sigma)}$, and if $a\in \calG(\Fbar)$ such that $i_-=i_+\cdot a$, then $\underline{I}\simeq\underline{I}_a$. Further the assignment $\underline{I}_a\mapsto E_{\chi}\cdot a$ is well-defined and $\underline{I}_a\simeq\underline{I}_{a'}$ if and only if $E_{\chi}\cdot a = E_{\chi}\cdot a'$ (\cite{PWZ2}, 3.10). This construction can be made functorial and induces an isomorphism of stacks $\calG_{\Fp}\mathtt{-Zip}^{\chi}\otimes \Fbar \simeq \left[ E_{\chi}\setminus \calG_{\Fbar}\right]$. (\cite{PWZ2}, 3.11.)
   
  Let $\calB_-$ be the Borel subgroup of $\calG_{\calO}$ which is opposite to $\calB$ with respect to $\calT$, let $B_-$ be its reduction to $\calG_{\Fbar}$, and let $y:=\dot{w}_{0,J}\dot{w}_0$. Then the triple $(B_-,T,\bar{y})$ is a frame for the zip datum $(\calG_{\Fbar},P_+,P_-^{(\sigma)},\sigma)$ in the sense of (\cite{PWZ1}, 3.6.). 
  Let $(W,S_-)$ be the Weyl group with respect to the pair $(B_-,T)$. We need to compare it to $(W,S)$, as for example explained in (\cite{PWZ1}, \S 2.3.): There is a unique isomorphism $\delta\colon (W,S)\to (W,S_-)$ of coxeter groups which is induced from an inner automorphism $\mathrm{int}(g)$, where $g\in\calG(\Fbar)$ such that $gBg^{-1}=B_-$ and $gTg^{-1}=T$. Since in our case we may choose $g=\bar{\dot{w}}_0$, we see that $\delta(w)=w_0ww_0$ for $w\in W$. Applying the results of (\cite{PWZ1}, \S6, \S7), which are formulated for the Weyl group $(W,S_-)$, we see that the assignment
  \[
    ^JW\longrightarrow E_{\chi}\setminus\calG_{\Fbar},\quad w\longmapsto O^w:= E_{\chi}\cdot(\bar{y}\overline{\dot{\delta(w)}})=E_{\chi}\cdot(\bar{\dot{w}}_{0,J}\bar{\dot{w}}\bar{\dot{w}}_0) 
  \]
  is bijective, and that $O^{w'}$ is contained in the closure of $O^w$ if and only if $w'\preceq w$ (we defined $\preceq$ in section \ref{EOStratSec}). 
  
  Altogether we obtain the homeomorphism $\calG_{\Fp}\mathtt{-Zip}^{\chi}(\Fbar)\simeq{^JW}$ from \ref{StackTopSpace}: It maps the ismomorphism class of $\underline{I}_a$ to the unique $w\in{^JW}$ such that $a\in O^w$.
\end{blank}

\begin{remark}
  \label{OrbitRepRem}
  \begin{enumerate}[(1)]
    \item We have already used the fact that for every $n\in N_{\calG}(\calT)(\Fbar)$ and $t\in \calT(\Fbar)=T$ we have $E_{\chi}\cdot nt=E_{\chi}\cdot n$. This is a consequence of Lang's Theorem, since  $\{(t',\sigma(t'))\mid t'\in T\}\subseteq E_{\chi}$.\\
    We will make liberal use of this property, for example we also use the set of representatives $\{\overline{\dot{w_{0,J}ww_0}}\mid w\in{^JW}\}$ for the set of $E_{\chi}$-orbits.
    \item We have $\dot{w}_{0,J}\in\calM(\calO)$ and therefore $(\bar{\dot{w}}_{0,J},\overline{\sigma(\dot{w}_{0,J})})\in E_{\chi}$ (independent of the choice of the lift for $w_{0,J}$). So $\{\overline{\dot{ww_0\sigma(w_{0,J})}}\mid w\in{^JW}\}$ is also a set of representatives for the $E_{\chi}$-orbits in $\calG_{\Fbar}$.
  \end{enumerate}
\end{remark}

\begin{definition}
\label{TrClassDef}
  For an element $g\in G(L)$ let $[[g]]$ be the set of all $K$-$\sigma$-conjugates of elements in $K_1gK_1$, i.e. 
  \[
    [[g]]:=\{hg'\sigma(h)^{-1}\mid g'\in K_1gK_1\}.
  \]
  If $g\in K\mu(p)K$, we will also identify $[[g]]$ with its image in $\calC(\calG,\mu)$. 
\end{definition}

Since $K_1$ is a normal subgroup of $K$, the sets $[[g]]$ form a decomposition of $G(L)$ (resp. $\calC(\calG,\mu)$) into equivalence classes. These classes have been defined and studied in greater generality by Viehmann in \cite{Vi1}.

\begin{proposition}
\label{ZetaFacProp}
  Define $\tilde{\zeta}\colon\calC(\calG,\mu)\to{^JW}$ as the composition
  \begin{align*}
    \calC(\calG,\mu) & \longrightarrow E_{\chi}\setminus\calG_{\Fbar}\qquad\simeq\quad{^JW}.\\
    \langle h_1\mu(p)h_2\rangle & \longmapsto E_{\chi}\cdot (\sigma^{-1}(\bar{h}_2)\bar{h}_1)
  \end{align*}
  The following hold:
  \begin{enumerate}[(i)]
    \item The map $\tilde{\zeta}$ is well-defined.
    \item We have the identity $\zeta=\tilde{\zeta}\circ\gamma$.
    \item For $g,g'\in K\mu(p)K$ we have $\tilde{\zeta}(\langle g\rangle)=\tilde{\zeta}(\langle g'\rangle)\Longleftrightarrow [[g]]=[[g']]$.
  \end{enumerate}  
\end{proposition}

\begin{proof}
  For every root $\alpha\in\Phi$ let $\calU_{\alpha}\colon\bbG_{a,\calO}\to\calG_{\calO}$ be the associated root group. For every $\alpha\in\Phi$ and $\lambda\in X_*(\calT)$ we have the relation
  \[
    \lambda(p)\calU_{\alpha}(x)\lambda(p)^{-1}=\calU_{\alpha}(p^{\langle\alpha,\lambda\rangle}x)\quad\text{for all } x\in L. \tag{*}
  \] 
  In particular, if $\langle\alpha,\lambda\rangle>0$ then $\lambda(p)\calU_{\alpha}(\calO)\lambda(p)^{-1}\subseteq K_1$.
  \begin{enumerate}
    \item[(i)] To show that $\tilde{\zeta}$ is well-defined it is clearly enough to show that the orbit $E_{\chi}\cdot (\sigma^{-1}(\bar{h}_2)\bar{h}_1)$ is independent of the choice of $h_1$ and $h_2$ (see also \cite{HL}, where a similar result is proved in Lemma 4.1.). So let $h_1', h_2'\in K$ $(i=1,2)$ such that $h_1\mu(p)h_2=h_1'\mu(p)h_2'$. We define
    \[
      c_1:=h_1^{-1}h_1',\quad \tilde{c}_2:=h_2(h_2')^{-1}, \quad c_2:=\sigma^{-1}(\tilde{c}_2).
    \]
    Then $\bar{c}_2(\sigma^{-1}(\bar{h}_2')\bar{h}_1')\bar{c}_1^{-1}=\sigma^{-1}(\bar{h}_2)\bar{h}_1$, so it suffices to show that $(\bar{c}_2,\bar{c}_1)\in E_{\chi}$.
  
    Since $c_1=\mu(p)\tilde{c}_2\mu(p)^{-1}$, we have $c_2\in K_{\chi}:= K\cap \chi(p)^{-1}K\chi(p)$. This is the stabilizer of two points in the Bruhat-Tits building of $\calG_L$ (in fact, it is even a parahoric subgroup of $\calG(L)$, since $\chi$ is minuscule), so we may apply the structure theory for these groups to $K_{\chi}$:\\
    For all $\alpha\in\Phi$ define $\calU_{\alpha}^{\chi}:=\calU_{\alpha}(L)\cap K_{\chi}$. From $(*)$ we see that $\calU_{\alpha}^{\chi} = \calU_{\alpha}(p^a\calO)$ with $a=\max\{0,-\langle\alpha,\chi\rangle\}$. Note that $\calU_{\alpha}^{\chi}\subseteq K_1$ if $\langle\alpha,\chi\rangle<0$. Now it follows from (\cite{Ti1}, 3.1.) that we may write $c_2=u_-u_+m$, where 
    \[
      u_+\in\prod_{\langle\alpha,\chi\rangle>0}\calU_{\alpha}^{\chi}\subseteq\calU_+(\calO),\quad u_-\in\prod_{\langle\alpha,\chi\rangle<0}\calU_{\alpha}^{\chi}\subseteq K_1,\quad m\in\calM(\calO)
    \]
    (here we use that $N_{\calG}\calT(L)\cap K_{\chi}\subseteq \calM(\calO)$ which can be easily checked). Thus we have $\bar{c}_2\in P_+$, with Levi component $\bar{m}$. Using the equation $\sigma^{-1}(c_1)=\chi(p)c_2\chi(p)^{-1}$ we now see that $c_1=u_+'u_-'\sigma(m)$ with $u_-'\in\sigma(\calU_-(\calO))$ and
    \[
      u_+'\in\sigma\Big(\prod_{\langle\alpha,\chi\rangle>0}\chi(p)\calU_{\alpha}^{\chi}\chi(p)^{-1}\Big)\subseteq K_1.
    \]
    Hence $\bar{c}_1\in P_-^{(\sigma)}$ and has Levi component $\overline{\sigma(m)}=\sigma(\bar{m})$, which shows that $(\bar{c}_2,\bar{c}_1)\in E_{\chi}$.
    \item[(iii)] Now we investigate the fibers of $\tilde{\zeta}$ (cf. the proof of Thm. 1.1.(1) in \cite{Vi1}). Let $\langle g\rangle, \langle g'\rangle\in\calC(\mu,K)$. Since everything only depends on the $K$-$\sigma$-conjugacy classes we may suppose that $g=h\mu(p)$ and $g'=h'\mu(p)$ for $h,h'\in K$. By definition of $\tilde{\zeta}$ we then have to show that 
    \[
      E_{\chi}\cdot \bar{h}=E_{\chi}\cdot \bar{h}'\ \Longleftrightarrow\ [[h\mu(p)]]=[[h'\mu(p)]].
    \]
    The implication "$\Longleftarrow$" follows directly from the proof of (i). Conversely, let $(p_+,p_-)\in E_{\chi}$ such that $p_+\bar{h}p_-^{-1}=\bar{h}'$. We may choose 
    \[
      m\in\calM(\calO),\quad u_+\in\calU_+(\calO),\quad u_-\in\sigma(\calU_-(\calO))
    \]
    such that $p_+=\bar{u}_+\bar{m}$ and $p_-=\bar{u}_-\sigma(\bar{m})$. By $(*)$ we have $\mu(p)^{-1}u_-^{-1}\mu(p)\in K_1$ and $\mu(p)\sigma(u_+)\mu(p)^{-1}\in K_1$. Thus, using the fact that $\sigma(m)^{-1}$ commutes with $\mu(p)$, we find that
    \begin{align*}
      [[h'\mu(p)]] & =[[u_+mh\sigma(m)^{-1}u_-^{-1}\mu(p)]]=[[u_+mh\sigma(m)^{-1}\mu(p)]]\\ 
                   & = [[mh\sigma(m)^{-1}\mu(p)\sigma(u_+)]]=[[mh\sigma(m)^{-1}\mu(p)]] \\
                   & =[[mh\mu(p)\sigma(m)^{-1}]]=[[h\mu(p)]].
    \end{align*}
    \item[(ii)] Finally we check that $\zeta=\tilde{\zeta}\circ\gamma$. Let $\underline{I}$ be the $\calG_{\Fp}$-zip associated to $\chi$ in Constr. \ref{EOConstr}, which defines $\zeta$. Consider a point $x\in\calS$. Let $\beta\colon (\Lambda^*_{\calO},s_{\calO})\to(\bbD_x,s_{\mathrm{cris},x})$ be an isomorphism, and let $F_{\beta}=g_{\beta}^{\vee}\circ(1\otimes\sigma)$ for $g_{\beta}\in K\mu(p)K$, then $\gamma(x)=\langle g_{\beta}\rangle$. We may suppose that $g_{\beta}=h\mu(p)$ for some $h\in K$ (compare Lemma \ref{LinearizationLem}). In view of the classification of $\calG_{\Fp}$-zips over $\Fbar$ and the definitions of $\zeta$ and $\tilde{\zeta}$ we then have to show that the pullback $\underline{I}_x$ is isomorphic to the standard zip $\underline{I}_{\bar{h}}$. 
  
    Let $\Lambda^*_{\Fbar}=\Lambda^*_{\chi,0}\oplus\Lambda^*_{\chi,1}$ be the weight decomposition of with respect to the action of $\mu$ on $\Lambda^*_{\Fbar}$ (using, as always, the representation $(\cdot)^{\vee}$), and likewise for $\mu$. Keeping the notations of \ref{EOConstr}, by definition we have $\underline{I}_x=(I_x,I_{+,x},I_{-,x},\iota_x)$, where
    \begin{align*}
      I_x:= & \mathbf{Isom}_{\Fbar}\big((\Lambda^*_{\Fbar},\bar{s}),\, (\overline{\calV^{\circ}}_x,\bar{s}_{\mathrm{dR},x})\big),\\
      I_{+,x}:= & \mathbf{Isom}_{\Fbar}\big((\Lambda^*_{\Fbar},\bar{s},\Lambda^*_{\Fbar}\supset\Lambda^*_{\chi,1}),\, (\overline{\calV^{\circ}}_x,\bar{s}_{\mathrm{dR},x},\overline{\calV^{\circ}}_x\supset\calC_x)\big),\\
      I_{-,x}:= & \mathbf{Isom}_{\Fbar}\big((\Lambda^*_{\Fbar},\bar{s},\Lambda^*_{\mu,0}\subset\Lambda^*_{\Fbar}),\, (\overline{\calV^{\circ}}_x,\bar{s}_{\mathrm{dR},x},\calD_x\subset\overline{\calV^{\circ}}_x)\big),
    \end{align*} 
    and $\iota_x\colon I_{+,x}^{(\sigma)}/U_+^{(\sigma)}\to I_{-,x}/U_-^{(\sigma)}$ is given as follows: Fix an element $\eta_-\in I_{-,x}$. Then for each $\eta_+\in I_{x,+}$ the image $\iota_x(\eta_+^{(\sigma)}U_+^{(\sigma)})$ is the $U_-^{(\sigma)}$-coset of the isomorphism
    \begin{align*}
      \Lambda^*_{\Fbar}=\Lambda^*_{\mu,0}\oplus\Lambda^*_{\mu,1} & \stackrel{\eta_+^{(\sigma)}}{\longrightarrow}\eta_+^{(\sigma)}(\Lambda^*_{\mu,0})\oplus\calC^{(\sigma)}
                                                                   \simeq(\overline{\calV^{\circ}}^{(\sigma)}/\calC^{(\sigma)})\oplus \calC^{(\sigma)}\\
                                                                 & \stackrel{\phi_0\oplus\phi_1}{\longrightarrow} \calD\oplus (\overline{\calV^{\circ}}/\calD)
                                                                   \simeq\calD\oplus\eta_-(\Lambda^*_{\mu,1})=\overline{\calV^{\circ}}
    \end{align*}
    (here we have omitted the subscripts), and this is in fact independent of the choice of $\eta_-$.\\
    Let $(\overline{\bbD_x},\overline{F},\overline{V})$ be the reduction mod $p$ of $\bbD_x$, this is the contravariant Dieudonn{\'e} space of $\calA_x[p]$ (compare the proof of Lemma \ref{LinearizationLem}). We use the canonical isomorphism $(\overline{\calV^{\circ}}_x,\bar{s}_{\mathrm{dR},x})\cong(\overline{\bbD_x},s_{\mathrm{cris},x})$. By the result of Oda (\cite{Od1}, 5.11.) we know that $\calC_x$ and $\calD_x$ correspond to the subspaces $\ker(\overline{F})=\im(\overline{V})$ and $\ker(\overline{V})=\im(\overline{F})$ of $\overline{\bbD_x}$ respectively, and the isomorphisms $\phi_0$ and $\phi_1$ get identified with the maps
    \[
      \left(\overline{\bbD_x}/\ker(\overline{F})\right)^{(\sigma)}\ \stackrel{\overline{F}^{\mathrm{lin}}}{\longrightarrow}\ \im(\overline{F}),\quad \im(\overline{V})^{(\sigma)}\ \stackrel{(\overline{V}^{-1})^{\mathrm{lin}}}{\longrightarrow}\ \overline{\bbD_x}/\ker(\overline{V}).  
    \]
    Note that $V^{\mathrm{lin}}\colon\bbD_x\to\bbD_x^{(\sigma)}$ is given by $p\cdot (F^{\mathrm{lin}})^{-1}$, so we have the identity $(\beta^{(\sigma)})^{-1}\circ V^{\mathrm{lin}}\circ\beta=p\cdot (g_{\beta}^{\vee})^{-1}$. Denote by $\mathrm{pr}_i$ the projections on the factors of the decomposition $\Lambda^*_{\Fbar}=\Lambda^*_{\mu,0}\oplus\Lambda^*_{\mu,1}$. The reduction mod $p$ of the map on $\Lambda^*_{\calO}$ given by $\mu(p)^{\vee}$ is exactly $\mathrm{pr}_0$, and the reduction of the map $p\cdot(\mu(p)^{\vee})^{-1}$ is $\mathrm{pr}_1$, so we have the commutative diagrams
    \[
      \begin{xy}
        \xymatrix{
          \overline{\bbD_x}^{(\sigma)}\ar[rr]^{\overline{F}^{\mathrm{lin}}} & & \overline{\bbD_x}\\
          \Lambda^*_{\Fbar}\ar[rr]^{\bar{h}^{\vee}\circ\mathrm{pr}_0} \ar[u]^{\bar{\beta}^{(\sigma)}} & & \Lambda^*_{\Fbar}\ar[u]_{\bar{\beta}}
        }
      \end{xy}
      \qquad
      \quad
      \begin{xy}
        \xymatrix{
          \overline{\bbD_x}\ar[rr]^{\overline{V}^{\mathrm{lin}}} & & \overline{\bbD_x}^{(\sigma)}\\
          \Lambda^*_{\Fbar}\ar[rr]^{\mathrm{pr}_1\circ(\bar{h}^{\vee})^{-1}} \ar[u]^{\bar{\beta}} & & \Lambda^*_{\Fbar}\ar[u]_{\bar{\beta}^{(\sigma)}}
        }
      \end{xy},
    \]
    and in particular $\bar{\beta}^{-1}(\ker(\overline{F}))=(1\otimes\sigma^{-1})(\Lambda^*_{\mu,1})=\Lambda^*_{\chi,1}$ and $\bar{\beta}^{-1}(\im(\overline{F}))=\bar{h}^{\vee}(\Lambda^*_{\mu,0})$. This implies that $\bar{\beta}\in I_{x,+}$ and $\bar{\beta}\cdot\bar{h}=\bar{\beta}\circ\bar{h}^{\vee}\in I_{x,-}$. Putting things together and setting $\calC_0:=\bar{\beta}^{(\sigma)}(\Lambda^*_{\mu,0})$ and $\calD_1:=\bar{\beta}(\bar{h}^{\vee}(\Lambda^*_{\mu,1}))$, from the diagrams above we obtain a commutative diagram 
    \[
      \begin{xy}
        \xymatrix{
          **[r] \overline{\calV^{\circ}}^{(\sigma)}\simeq\calC_0\oplus\calC^{(\sigma)}\ar[rrrr]^{\phi_0\oplus\phi_1} & & & & **[l] \calD\oplus\calD_1\simeq\overline{\calV^{\circ}}\\
          \Lambda^*_{\Fbar}\ar[rrrr]^{\bar{h}^{\vee}} \ar[u]^{\bar{\beta}^{(\sigma)}} & & & & \Lambda^*_{\Fbar}\ar[u]_{\bar{\beta}}
        }
      \end{xy}
    \]
    which shows that $\iota_x(\bar{\beta}^{(\sigma)}U_+^{(\sigma)})=(\bar{\beta}\circ\bar{h}^{\vee})U_-^{(\sigma)}$. So we have indeed  $\underline{I}_x\sim\underline{I}_{\bar{h}}$, which concludes the proof.
  \end{enumerate}
\end{proof}

\subsection{Group theoretic criteria and the $\mu$-ordinary locus}
\label{GrThCritSubsec}

Let us summarize the results of the last subsection: The Newton strata $\calN^b$ and Ekedahl-Oort strata $\calS^w$ are given as the fibers of the maps $\theta$ and $\zeta$ respectively. We have a commutative diagram
\[
  \begin{xy}
     \xymatrix{
                                                                & &                                                             & B(G,\mu)\\
       \calS\ar[rrru]^{\theta}\ar[rr]^{\gamma}\ar[rrrd]_{\zeta} & & \calC(\calG,\mu)\ar[ru]_{\tilde{\theta}}\ar[rd]^{\tilde{\zeta}} &               \\
                                                                & &                                                             & {^JW}
     }
  \end{xy}
\]
For $b\in B(G,\mu)$ we have $\tilde{\theta}^{-1}(\{b\})=\calC(\calG,\mu)\cap b$ and, in view of \ref{ZipClass}, Remark \ref{OrbitRepRem} and Prop. \ref{ZetaFacProp}, for $w\in{^JW}$ we have $\tilde{\zeta}^{-1}(\{w\})=[[\dot{w}_{0,J}\dot{w}\dot{w}_0\mu(p)]]=[[\dot{w}\dot{w}_0\sigma(\dot{w}_{0,J})\mu(p)]]$, compare (\cite{Vi1}, Thm. 1.1.(1)). 
\begin{definition}
\label{TrRepDef}
  For $w\in{^JW}$ we define $\tilde{w}:=\dot{w}\dot{w}_0\sigma(\dot{w}_{0,J})\mu(p)$.
\end{definition}

We note a few consequences for the comparison of the two stratifications:
\begin{enumerate}[(1)]
  \item For $b\in B(G,\mu)$ and $w\in{^JW}$ we have the following necessary criterion for the corresponding strata to intersect:
  \begin{center}
    If $\calN^b\cap\calS^w\neq\emptyset$, then $(\calC(\calG,\mu)\cap b)\cap [[\tilde{w}]]=b\cap [[\tilde{w}]]\neq\emptyset$.
  \end{center}
  If the map $\gamma$ is surjective, then this criterion is also sufficient.
  \item Let $b\in B(G,\mu), w\in{^JW}$. If $[[\tilde{w}]]\subseteq \calC(\calG,\mu)\cap b$ (resp. $\supseteq$, resp. $=$), then $\calS^w\subseteq \calN^b$ (resp. $\supseteq$, resp. $=$).
  \item Let $I:=\{g\in K\mid \bar{g}\in B\}\subseteq K$, this is an Iwahori subgroup of $\calG(L)$. By (\cite{Vi1}, Thm. 1.1.(2)), for every $w\in{^JW}$ any element of $I \tilde{w}I$ is $I$-$\sigma$-conjugate to an element in $K_1 \tilde{w} K_1$, which means that in particular the sets $I \tilde{w}I$ and $[[\tilde{w}]]$ have the same image in $\calC(\calG,\mu)$. Hence we may replace $[[\tilde{w}]]$ by $I \tilde{w} I$ in (1) and (2).\\
  In particular the question whether Newton strata and Ekedahl-Oort strata intersect is closely related to the study of affine Deligne-Lusztig varieties: Let $w\in{^JW}$ and $b\in B(G,\mu)$, and choose a $b_0\in b$. The affine Deligne-Lusztig variety of $\tilde{w}$ and $b_0$ is defined as $X_{\tilde{w}}(b_0):=\{gI\in G(L)/I\mid g^{-1}b_0\sigma(g)\in I\tilde{w}I\}$. It is nonempty if and only if $b$ and $I\tilde{w}I$ intersect. So if $\calN^b\cap\calS^w\neq\emptyset$ then $X_{\tilde{w}}(b_0)\neq\emptyset$, and the converse is true if $\gamma$ is surjective.
\end{enumerate}

As an application of these principles we can now prove our main results on the $\mu$-ordinary locus: 

\begin{proposition}
\label{OrdEqProp1}
  We have the equalities $K\mu(p)K\cap[\mu(p)]=\langle\mu(p)\rangle=[[\mu(p)]]$.
\end{proposition}

In fact, this statement holds true in a more general setup. We will restate and prove this proposition in the next section. It implies the main theorems already stated in the introduction:

\begin{theorem}[cf. Thm. \ref{MainThm2}]
\label{OrdLocThm}
  The $\mu$-ordinary locus in $\calS$ is equal to the open Ekedahl-Oort stratum $\calS^{w_{\mathrm{max}}}$, in particular it is open and dense. Further, for any two $\Fbar$-valued points $x,x'$ in the $\mu$-ordinary locus there is an isomorphism $(\bbD_x,s_{\mathrm{cris},x})\simeq(\bbD_{x'},s_{\mathrm{cris},x'})$.
\end{theorem}

\begin{proof}
  The $\mu$-ordinary locus is by definition equal to $\calN^{b_{\mathrm{max}}}$, and we have $b_{\mathrm{max}}=[\mu(p)]$, see Remark \ref{NewtParRem}. On the other hand, as $w_{\mathrm{max}}=w_{0,J}w_0$, we have $[[\tilde{w}_{\mathrm{max}}]]=[[\mu(p)]]$. So the equality of the strata follows from Prop. \ref{OrdEqProp1} and the observation (2) above. By Remark \ref{EOPropRem}(3), $\calS^{w_{\mathrm{max}}}$ is open and dense in $\calS$, so the same holds for the $\mu$-ordinary locus. Since Prop. \ref{OrdEqProp1} also asserts that we have the equality $\calN^{b_{\mathrm{max}}}=\gamma^{-1}(\{\langle\mu(p)\rangle\})$, the last statement follows from Corollary \ref{LinearizationCor}.
\end{proof}

\begin{corollary}[cf. Thm. \ref{MainThm1}]
\label{OrdLocCor}
  The $\mu$-ordinary locus in $\scrS\otimes\kappa(v)$ is open and dense.
\end{corollary}

\section{Ordinary loci for reductive unramified groups}
\label{GroupResSec}

In the final section of this paper we prove Prop. \ref{OrdEqProp1} in a more general setup which also includes the function field case. We modify our notation accordingly:

\begin{notation}
  Let $\Fq$ be a finite extension of $\Fp$, and let $F$ be either $\Qq$ or the field $\Fq((t))$ of Laurent series.
We denote by $L$ the completion of the maximal unramified extension of $F$, i.e. either $L=\Frac(W(\Fbar))$ or $L=\Fbar((t))$. Here as before $\Fbar$ is an algebraic closure of $\Fp$. Let $\calO\subseteq L$ and $\calO_F\subseteq F$ be the valuation rings respectively, let $\epsilon$ be a uniformizing element of $\calO_F$ (and hence of $\calO$), e.g. $\epsilon=p$ in the case of mixed characteristics or $\epsilon=t$ in the equicharacteristic case. In this section we denote by $\sigma$ the Frobenius map $a\mapsto a^q$ of $\Fbar$ over $\bbF_q$, and also the corresponding maps on $\calO$ and on $L$. 

  Let $\calG$ be a connected, reductive group over $\calO_F$ (it is then quasisplit and split over a finite unramified extension of $\calO_F$). Fix a Borel pair $(\calB,\calT)$ defined over $\calO_F$. As before, we obtain the groups $X_*(\calT), X^*(\calT)$ of cocharacters and characters of $\calT$ over $\calO$, the roots resp. positive roots $\Phi^*\subset\Phi\subset X^*(T)$ with respect to $\calT$, and the Weyl group $(W,S)$, all equipped with an action of $\sigma$.

As in section \ref{StratCompSec} let $K:=\calG(\calO)$ and $K_1:=\{g\in K\mid \bar{g}=1\}$, and for $g\in\calG(L)$ define the sets $[g]$, $\langle g\rangle$ and $[[g]]$ as before. Let $I:=\{g\in K\mid \bar{g}\in \calB(k)\}$, this is an Iwahori subgroup of $\calG(L)$.

\end{notation}

The formulation of Prop. \ref{OrdEqProp1} in this context reads as follows: 

\begin{proposition}
\label{OrdEqProp2}
  Let $\mu\in X_*(\calT)$ be a dominant cocharacter. Then for $g\in \calG(L)$ the following are equivalent:
  \begin{enumerate}[(i)]
    \item $g\in [\mu(\epsilon)]\cap K\mu(\epsilon)K$,
    \item $g\in \langle\mu(\epsilon)\rangle$,
    \item $g\in [[\mu(\epsilon)]]$.
  \end{enumerate}
\end{proposition}

Note that $\mu$ is not assumed to be minuscule. 

\begin{remark}
\label{MoonenCompRem}
  This proposition should be seen as a generalization of (\cite{Mo2}, Thm. 1.3.7. resp. Thm. 3.2.7.):  If $\calG$ arises from a PEL-type Shimura datum, then the element $\mu(\epsilon)\in\calG(L)$ takes the place of the $p$-divisible group $\underline{X}^{\mathrm{ord}}$ defined in \cite{Mo2}.
\end{remark}

We will prove Prop. \ref{OrdEqProp2} throughout the rest of the section. Of course, the implications $(ii)\Rightarrow(i)$ and $(ii)\Rightarrow(iii)$ are trivial.\\

  The implication $(iii)\Rightarrow(ii)$ follows from the property that every element in the double coset $I\mu(\epsilon) I$ is $\sigma$-conjugate to $\mu(\epsilon)$ by an element of $I$. This is well-known, for example it is a consequence of the fact that, with the conventions of \cite{GHN}, the element $\mu\in \widetilde{W}=X_*(\calT)\rtimes W$ is a fundamental $(\emptyset,1,\sigma)$-alcove (see loc.cit., Thm. 3.3.1., Prop. 3.4.3.). So if $g\in [[\mu(\epsilon)]]$, say $g=hh_1\mu(\epsilon)h_1'\sigma(h)^{-1}$ for some $h\in K$ and $h_1,h_1'\in K_1$, then in particular $h_1\mu(\epsilon)h_1'\in I\mu(\epsilon)I$ and so there is an $i\in I$ such that $h_1\mu(\epsilon)h_1'=i\mu(\epsilon)\sigma(i)^{-1}$ and thus $g=(hi)\mu(\epsilon)\sigma(hi)^{-1}\in\langle\mu(\epsilon)\rangle$.\\

  To show the remaining implication we will use the Hodge-Newton decomposition for affine Deligne-Lusztig sets in the affine Grassmannian, which was first formulated for unramified groups by Kottwitz and later generalized by Mantovan and Viehmann. We briefly recall the main statement: Let $\lambda\in X_*(\calT)$ be a cocharacter (which is usually taken to be dominant) and let $b_0\in\calG(L)$, then the affine Deligne-Lusztig set associated to these elements is defined as
\[
  X_{\lambda}^{\calG}(b_0):=\{g\in\calG(L)/K\mid g^{-1}b_0\sigma(g)\in K\lambda(\epsilon)K\}.
\]
Let $\calM\subseteq\calG$ be a Levi subgroup defined over $\calO_F$ such that $\calT\subseteq\calM$ (note that this differs from the $\calM$ considered in section \ref{StratCompSec}), then $\calM$ is again a connected reductive group. Consider the Borel pair $(\calB\cap\calM, \calT)$. Then $\Phi_{\calM}\subseteq\Phi$, $\Phi_{\calM}^+=\Phi^+\cap\Phi_{\calM}$ and the Weyl group of $\calM$ is of the form $(W_{S_{\calM}},S_{\calM})$ for some subset $S_{\calM}\subseteq S$. We have the Newton map and Kottwitz map for $\calM$
\[
  \nu_{\calM}\colon B(\calM)\longrightarrow(W_{\calM}\setminus X_*(\calT)_{\bbQ})^{\langle\sigma\rangle}, \qquad \kappa_{\calM}\colon B(\calM)\longrightarrow \pi_1(\calM)_{\langle\sigma\rangle}.
\]
For every $\lambda\in X_*(\calT)$ and $b_0\in\calM(L)$ there is the analogous Deligne-Lusztig set 
\[
  X_{\lambda}^{\calM}(b_0)=\{m\in\calM(L)/\calM(\calO)\mid m^{-1}b_0\sigma(m)\in \calM(\calO)\lambda(\epsilon)\calM(\calO)\},
\]
and a natural map $X_{\lambda}^{\calM}(b_0)\to X_{\lambda}^{\calG}(b_0)$, which is clearly injective. 

Let $V:=X_*(\calT)\otimes_{\bbZ}\bbR$, it carries an action of $W$ and of $\sigma$. Let $V_{\calM}\subseteq V$ be the subspace of elements which are invariant under the action of $W_{\calM}$ and $\sigma$, and let 
\[
  V_{\calM}^+:=\{v\in V_{\calM}\mid \langle\alpha,v\rangle >0\text{ for all }\alpha\in\Phi^+\setminus\Phi_{\calM}^+\}.
\]
Fix an $n\in\bbN$ such that $\sigma^n$ acts as the identity on $V$. The composition of the two maps 
\[
 v\longmapsto \frac{1}{n}\sum_{i=0}^{n-1}\sigma^i(v),\quad  \quad v\longmapsto \frac{1}{|W_{\calM}|}\sum_{w\in W_{\calM}}w(v)
\]
gives a projection map $V\twoheadrightarrow V_{\calM}$, which induces an isomorphism $\pi_1(\calM)_{\langle\sigma\rangle}\otimes_{\bbZ}\bbR\simeq V_{\calM}$. Let $\pi_1(\calM)_{\langle\sigma\rangle}^+$ be the subset of elements whose image under the resulting map $\pi_1(\calM)_{\langle\sigma\rangle}\to V_{\calM}$ lies in $V_{\calM}^+$.

\begin{proposition}[\cite{MV}, Thm. 6]
\label{HNDecProp}
  Let $\lambda\in X_*(\calT)$ be $\calG$-dominant, let $b_0\in\calM(L)$ such that $\kappa_{\calM}([b_0]_{\calM})$ equals the image of $\lambda$ in $\pi_1(\calM)_{\langle\sigma\rangle}$. If $\kappa_{\calM}([b_0]_{\calM})\in\pi_1(\calM)_{\langle\sigma\rangle}^+$ and the $\calM$-dominant representative of $\nu_{\calM}([b_0]_{\calM})\in (W_{\calM}\setminus X_*(\calT)_{\bbQ})^{\langle\sigma\rangle}$ is also $\calG$-dominant, then the natural map $X_{\lambda}^{\calM}(b_0)\hookrightarrow X_{\lambda}^{\calG}(b_0)$ is an isomorphism.
\end{proposition}

Now we show the remaining implication $(i)\Rightarrow (ii)$ of Prop. \ref{OrdEqProp2}: Consider an element $g\in [\mu(\epsilon)]\cap K\mu(\epsilon)K$. Then $g=h^{-1}\mu(\epsilon)\sigma(h)$ for some $h\in\calG(L)$, and we need to show that we may replace $h$ by some element of $K$. By definition, $h$ lies in the affine Deligne-Lusztig set $X_{\mu}^{\calG}(\mu(\epsilon))$. Let $\bar{\mu}:=\frac{1}{n}\sum_{i=0}^{n-1}\sigma^i(\mu)$, where as before $n\in\bbN$ is chosen such that $\sigma^n$ acts as the identity. Consider the subgroup
\[
  \calM:=\Cent_{\calG}(\bar{\mu}):=\Cent_{\calG}(n\cdot\bar{\mu})\subseteq\calG,
\]
this is a Levi subgroup (cf. \cite{SGA3}, Exp. XXVI, Cor. 6.10.) which is defined over $\calO_F$, as $n\cdot\bar{\mu}$ is $\sigma$-invariant. 

We claim that $\mu(\epsilon)$ is central in $\calM$: Indeed, we have 
\[
  \mathrm{Z}(\calM)=\bigcap_{\alpha\in\Phi_{\calM}}\ker(\alpha)\subseteq \calT
\]
(see \cite{SGA3}, Exp. XXII, Cor. 4.1.6.). Let $\alpha\in\Phi_{\calM}^+$. By definition of $\calM$ the cocharacter $n\cdot\bar{\mu}$ maps to the center of $\calM$, so we find that
\[
  0=\langle\alpha,n\cdot\bar{\mu}\rangle=\sum_{i=0}^{n-1}\langle\alpha,\sigma^i(\mu)\rangle.
\]
Since $\sigma$ acts on the set of dominant cocharacters, every summand in the upper equation is nonnegative, so they are all equal to zero. In particular, $\langle\alpha,\mu\rangle=0$ (for every $\alpha\in\Phi_{\calM}^+$), which implies that $\mu$ is also central in $\calM$.

Next, note that the pair $(\mu,\mu(\epsilon))$ satisfies the conditions of Prop. \ref{HNDecProp}: The $\calM$-dominant Newton vector of $[\mu(\epsilon)]_{\calM}$ is exactly $\bar{\mu}$ (cf. Remark \ref{NewtParRem}), which is $\calG$-dominant as well. Further, the image of $\kappa_{\calM}([\mu(\epsilon)]_{\calM})$ in $V_{\calM}$ is the projection of $\mu \in V$ to $V_{\calM}$, which is also equal to $\bar{\mu}$, and this lies in $V_{\calM}^+$ by definition of $\calM$. We have therefore the Hodge-Newton decomposition $X_{\mu}^{\calM}(\mu(\epsilon))\cong X_{\mu}^{\calG}(\mu(\epsilon))$, so there is an element $m\in\calM(L)$ such that
\[
  mK=hK\quad \text{and} \quad m^{-1}\mu(\epsilon)\sigma(m)\in\calM(\calO)\mu(\epsilon)\calM(\calO).
\]
Since $\mu(\epsilon)$ commutes with every element of $\calM(L)$, the last equation implies that $m^{-1}\sigma(m)\in\calM(\calO)$. As $\calM$ is a connected reductive group over $\calO$, a variant of Lang's theorem holds for $\calM(\calO)$ (see \cite{Vi1}, Lemma 2.1.), so we obtain an element $m'\in\calM(\calO)$ such that $(m')^{-1}\sigma(m')=m^{-1}\sigma(m)$. Let $c\in K$ such that $h=mc$, then altogether we have
\begin{align*}
  g=h^{-1}\mu(\epsilon)\sigma(h) & =c^{-1}(m^{-1}\mu(\epsilon)\sigma(m))\sigma(c)\\
                                 & =c^{-1}(m^{-1}\sigma(m)\mu(\epsilon))\sigma(c)\\
                                 & =c^{-1}((m')^{-1}\sigma(m')\mu(\epsilon))\sigma(c)\\
                                 & =c^{-1}(m')^{-1}\mu(\epsilon)\sigma(m')\sigma(c) \in\langle\mu(\epsilon)\rangle,
\end{align*}
which was to be shown. This concludes the proof of Prop. \ref{OrdEqProp2}. \hfill{$\Box$}

\end{document}